\documentclass[10pt,a4paper,final]{article}
\usepackage[latin1]{inputenc}
\usepackage{amsmath}
\usepackage{amsfonts}
\usepackage{amssymb}
\usepackage{amsthm}
\usepackage{makeidx}
\usepackage{graphicx}
\usepackage{a4wide}
\usepackage{todo}
\usepackage{latexsym}
\usepackage{cite}
\usepackage{mathrsfs}
\usepackage{bbm}
\usepackage{bbold}
\usepackage{float}
\usepackage{epsfig}
\usepackage{epstopdf}
\usepackage{hyperref}
\usepackage{mathtools}
\usepackage{listings}
\usepackage{pgf}
\usepackage[T1]{fontenc}
\usepackage[latin1]{inputenc}
\usepackage[english]{babel}
\usepackage{pifont,tabularx}
\usepackage{stmaryrd}
\usepackage{dsfont}
\usepackage{theoremref}
\usepackage{nicefrac}
\usepackage{enumerate}
\usepackage{comment}

\usepackage{tikz}
\usepackage{pgfplots}
\usepackage{pgfplotstable}
\pgfplotsset{compat=1.7}
\usepackage{caption}
\usepackage{subfigure}
\usepackage{caption}
\theoremstyle{plain} 
\newtheorem{thm}{Theorem}[section]

\newtheorem{cor}[thm]{Corollary}
\newtheorem{lem}[thm]{Lemma} 
\newtheorem{fact}[thm]{Fact} 

\newtheorem{obs}[thm]{Remark}

\theoremstyle{remark}

\newcommand{\E}{{\mathbb{E}}}   
\newcommand{\Prob}{{\mathbb{P}}}   

\title{Limit Laws for Critical Dispersion on Complete Graphs}

\author{
Umberto~De~Ambroggio\thanks{LMU Munich, Department of Mathematics, Theresienstr.~39, 80333 Munich, Germany. Email: \texttt{\{deambrog,makai,kpanagio,steibel\}@math.lmu.de}. Supported by ERC Grant Agreement 772606-PTRCSP.}
\and Tam\'as~Makai\footnotemark[1]
\and Konstantinos~Panagiotou\footnotemark[1]
\and Annika~Steibel\footnotemark[1]
}

\begin{document}

\maketitle

\begin{abstract}
We consider a synchronous process of particles moving on the vertices of a graph~$G$, introduced by Cooper, McDowell, Radzik, Rivera and Shiraga (2018). Initially,~$M$ particles are placed on a vertex of~$G$. In subsequent time steps, all particles that are located on a vertex inhabited by at least two particles jump independently to a neighbour chosen uniformly at random. The process ends at the first step when no vertex is inhabited by more than one particle; we call this (random) time step the \textit{dispersion time}.

In this work we study the case where $G$ is the complete graph on $n$ vertices and the number of particles is $M=n/2+\alpha n^{1/2} + o(n^{1/2})$, $\alpha\in \mathbb{R}$. This choice of $M$ corresponds to the critical window of the process, with respect to the dispersion time.
We show that the dispersion time, if rescaled by $n^{-1/2}$, converges in $p$-th mean, as $n\rightarrow \infty$ and for any $p \in \mathbb{R}$, to a continuous and almost surely positive random variable $T_\alpha$.
We find that $T_\alpha$ is the absorption time of a standard logistic branching process, thoroughly investigated by Lambert (2005), and we determine its expectation. In particular, in the middle of the critical window we show that $\mathbb{E}[T_0] = \pi^{3/2}/\sqrt{7}$, and furthermore we formulate explicit asymptotics when~$|\alpha|$ gets large that quantify the transition into and out of the critical window. We also study the (random) \emph{total number of jumps} that are performed by the particles until the dispersion time is reached. In particular, we prove that it centers around $\frac27n\ln n$ and that it has variations linear in $n$, whose distribution we can describe explicitly. 

~  
\end{abstract}

\section{Introduction}

The \textit{dispersion process} introduced by Cooper, McDowell, Radzik, Rivera and Shiraga \cite{CMRRS18} consists of particles moving on the vertices of a given graph $G$. A particle is said to be~\emph{happy} if there are no other particles occupying the same vertex and~\emph{unhappy} otherwise.
Initially,~$M \ge 2$ (unhappy) particles are placed on some vertex of $G$.  
Subsequently, at discrete time steps, all unhappy particles move \textit{simultaneously} and \textit{independently} to a neighbouring vertex selected uniformly at random, while the happy particles remain in place. The process terminates at the first time step at which all particles are happy; we call this (random) time step the \emph{dispersion~time}. 

It is clear that if the number of particles is small --  compared to the number of vertices in the graph -- then the dispersion time should be small as well. Intuitively, increasing the number of particles makes it more and more difficult for the particles to disperse quickly.
This transition from 'fast' to 'slow' dispersion is quite well-understood and sharp when the underlying graph is the complete graph on $n$ vertices with loops, in which case we write $T_{n,M}$ for the dispersion time started with $M$ particles at an arbitrary vertex (see further below for a precise definition of the model).
The typical order of $T_{n,M}$ changes rather abruptly around $M = n/2$. Indeed, if we write $M = M(n) =(1+\varepsilon)n/2 \in \mathbb{N}$ for some sequence $\varepsilon = \varepsilon(n) \in [-1,1]$, then in~\cite{CMRRS18} it was established that~$T_{n,M}$ is typically
\begin{itemize}
\setlength{\itemsep}{1pt}
    \item  at most logarithmic in $n$ when $\limsup_{n \to \infty} \varepsilon < 0$ and
    \item  at least exponential in  $n$ when $\liminf_{n\to\infty} \varepsilon > 0$.
\end{itemize}
The details of this apparent and abrupt transition from logarithmic to exponential time are obviously of great interest and were investigated further in~\cite{DMP23}, where the authors studied the typical order and the tails of $T_{n,M}$ when $\varepsilon = o(1)$, that is, when $M = n/2 + o(n)$. In this setting they showed that for any constant $C > 0$, if $\varepsilon \le -C n^{-1/2}$, then the process typically finishes in $\Theta(|\varepsilon|^{-1}\ln(\varepsilon^2 n))$ steps, while if $\varepsilon\ge C n^{-1/2}$, then a much larger number $\varepsilon^{-1}\exp(\Theta(\varepsilon^2 n))$ of steps is required. Moreover, within the \emph{critical window} corresponding to the range $|\varepsilon|=O(n^{-1/2})$, they showed that the process typically runs for $\Theta(n^{1/2})$ steps, making the  transition into and out of the critical window smooth, see also Figure~\ref{fig:smoothtransition}.

\begin{figure}
\centering
\begin{tikzpicture}
    \draw[->] (0,0) -- (9.7,0);
    
    \node[above] at (2,0.25) {$|\varepsilon|^{-1} \ln(\varepsilon^2n )$};
    \node[above] at (5,0.25) {$n^{1/2}$};
    \node[above] at (8,0.25) {$\varepsilon^{-1} \exp(\Theta(\varepsilon^2 n))$};
    
    \node[above] at (10.2,0.06) {\small $T_{n,M}$};
    \node[below] at (10.2,-0.06) { \small $\varepsilon$};
    
    \node[below] at (3.5,-0.2) { $-Cn^{-1/2}$ };
    \node[below] at (5,-0.3) { $0$ };
    \node[below] at (6.7,-0.2) { $Cn^{-1/2}$ };
    
    \draw (3.5,0.4) -- (3.5,-0.15);
    \draw (5,0.2) -- (5,-0.15);
    \draw (6.5,0.4) -- (6.5,-0.15);
\end{tikzpicture}

\captionsetup{width=.9\linewidth}
\caption{The typical order of $T_{n,M}$ when $M = (1+\varepsilon)n/2$ and  $|\varepsilon| = o(1)$. Note that $|\varepsilon^{-1}|\ln(\varepsilon^2 n)$ and $\varepsilon^{-1} \exp(\varepsilon^2 n)$ are in $\Theta(n^{1/2})$ when $|\varepsilon| = \Theta(n^{-1/2})$, and so the transition into and out of the critical window is smooth.}

\label{fig:smoothtransition}
\end{figure}
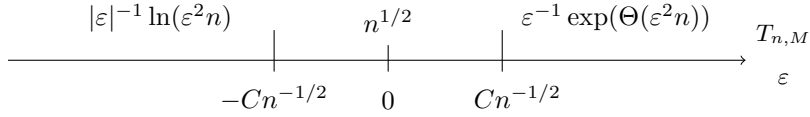

In this paper we will perform a fine analysis of the dispersion process within the critical window, that is, when $M = n/2 + O(n^{1/2})$. 
Our first main result establishes that the dispersion time, scaled by $n^{-1/2}$, converges in distribution to some continuous and almost surely positive random variable.
For a sequence of real-valued random variables $(Z_n)_{n\in \mathbb{N}}$ and a random variable~$Z$ we write $Z_n \overset{d}{\longrightarrow} Z$ to denote that the sequence $(Z_n)_{n\in \mathbb{N}}$ converges to $Z$ in distribution. 
\begin{thm}
\label{thm:main}
Let $\alpha \in \mathbb{R}$ and $M = M(n) = n/2 + \alpha n^{1/2} + o(n^{1/2}) \in \mathbb{N}$. Then there is a continuous and almost surely positive random variable $T_\alpha$ such that, as $n \to \infty$,
\[
    n^{-1/2} T_{n,M} \overset{d}{\longrightarrow} T_\alpha ~.
\]
\end{thm}
\noindent
Within the proof of Theorem~\ref{thm:main} we derive an explicit description of the distribution of $T_\alpha$, see also~\eqref{eq:distributionTalpha}. In order to specify it at this point we need to step back a bit and introduce some notation and present some facts about the process. Let us write $U_t$ for the (random) number of unhappy particles at the end of step~$t$, so that~$U_0 = M$, and let us fix some~$\delta > 0$.
As we will see in Section~\ref{sec:earlySteps},~$U_t$ drops rather quickly to~$\Theta(n^{1/2})$ particles. In particular, with probability at least $1-\delta$, after $t^* \sim \frac47\delta n^{1/2}$ steps we have that~$U_{t^*} \sim n^{1/2}/\delta$; here and everywhere else `$\sim$' will stand for~`$= (1+o(1))$' and asymptotic statements are, unless stated explicitly otherwise, with respect to~$n \to \infty$ and uniform in all other parameters.
After $t^*$ steps the process $(U_t)_{t \ge t^*}$ of unhappy particles starts fluctuating significantly, see Figure~\ref{fig:threeRunsAtCrit} for outcomes of a simulation study when $M=n/2$.
\begin{figure}[ht]
\centering
\begin{tikzpicture}[>=latex]
\begin{axis}[
        xmin = -220,
        xmax = 8300,
	xlabel=$t$,
        xlabel style={at={(axis cs:8321,-0.52)}, anchor=north},
        xtick={0,500, 1000,3000,5000,7000},
        xticklabels={$\phantom{'}0\phantom{'}$,$t'$,\phantom{'}1000\phantom{'},\phantom{'}3000\phantom{'},\phantom{'}5000\phantom{'},\phantom{'}7000\phantom{'}},
        axis x line*=bottom,
        ymin = -0.5,
        ymax = 10.5,
	ylabel=$\frac{U_t}{1000}$,
        ytick = {0,1,2,4,6,8},
        ylabel style={at={(axis cs:-480,9.8)}, anchor=west},
        axis y line*=left,
	grid=both,
	minor grid style={gray!25},
	major grid style={gray!40},
	width=0.95\linewidth,
	height=230pt,
        axis line style={-stealth, thick},
]
\addplot[line width=1.2pt,dotted]
	table[x=x,y expr={\thisrow{a}/1000},col sep=space]{Data/iterated_mean.csv};
\addlegendentry{Iterated Mean};
\addplot[line width=0.7pt,solid,color=blue]
	table[x=x,y expr={\thisrow{a}/1000},col sep=space]{Data/run1.csv};
\addplot[line width=0.7pt,solid,color=green]
	table[x=x,y expr={\thisrow{a}/1000},col sep=space]{Data/run2.csv};
\addplot[line width=0.7pt,solid,color=orange]
	table[x=x,y expr={\thisrow{a}/1000},col sep=space]{Data/run6.csv};
\node[draw,inner sep =3pt,fill=white] at (675,88)%
        { \begin{minipage}{4.5cm}
            \vspace{-12pt}
            \[
            \begin{split}
                \mathbb{E}[T_{n,M}]  & \sim \sqrt{\pi^{3}n/7} \approx 6655 
            \end{split}
            \]
            \end{minipage}
        };
\end{axis}
\end{tikzpicture}
\captionsetup{width=.9\linewidth}
\caption{Three sample runs of the dispersion process with $n=10^7$ and $M = n/2$, where we depict the number of unhappy particles $U_t$, divided by 1000, at each step $t$.
The trajectory is revealed only after $t'=500$, where $U_{t'} \approx 10^4 \approx 3n^{1/2}$ in all cases.
The dotted line represents the iterated mean of $U_t$, see also Lemma~\ref{lem:veryEarly}. For the asymptotics of $\mathbb{E}[T_{n,M}]$ see~\eqref{eq:ET0}.
}
\label{fig:threeRunsAtCrit}
\end{figure}
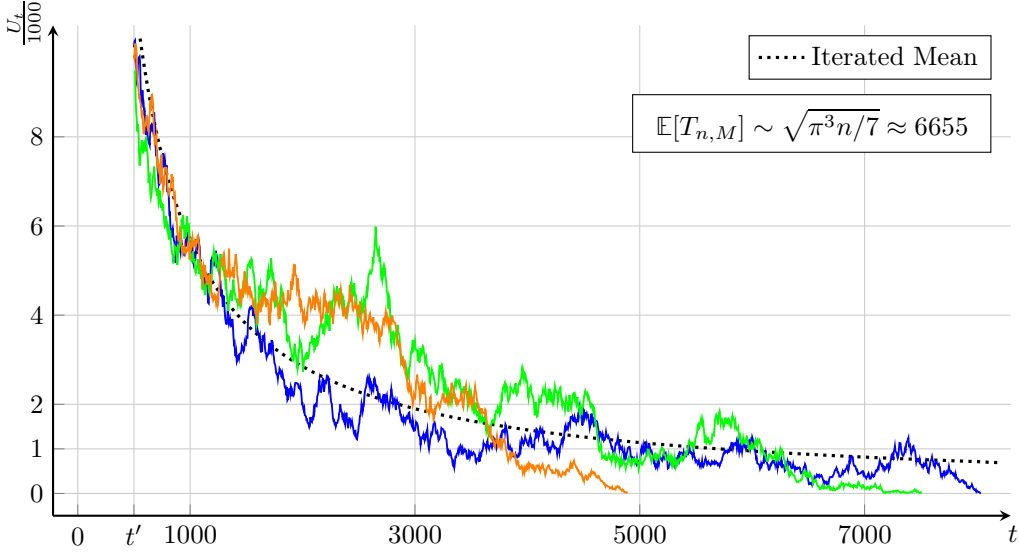
In order to get a grip on it, we scale time and space by a factor of~$n^{1/2}$ and establish that $(n^{-1/2}U_{t^* + \lfloor sn^{1/2} \rfloor})_{s \ge 0}$ converges weakly to a diffusion process. Here weak convergence denotes, as usual, convergence in $D([0,T],\mathbb{R})$ for all $T<\infty$, where $D([0,T],\mathbb{R})$ is the space of all c\`adl\`ag functions from $[0,T]$ to $\mathbb{R}^{d}$ that we equip with the Skorokhod metric.
\begin{lem}
\label{thm:processconvergence}
Let $\alpha \in \mathbb{R}$ and $M = M(n) = n/2 + \alpha n^{1/2} + o(n^{1/2}) \in \mathbb{N}$. Let $\delta > 0$ and
\[
    T_{n,M,\delta}
    \coloneqq \inf \{ t > 0 : U_t \le n^{1/2}/\delta\}
\]
be the first step at which there are at most $n^{1/2}/\delta$ unhappy particles. Then, as $n\to\infty$, weakly 
\[
    \Big(n^{-1/2}\,U_{T_{n,M,\delta} + \lfloor s n^{1/2} \rfloor}\Big)_{s\ge 0} \to X,
\]
where $X$ is a logistic branching process. In particular, if we denote by $B$ a standard Brownian motion, then $X$ uniquely satisfies the stochastic differential equation
\begin{equation}
\label{eq:SDEFellerDiffusionX}
    dX_{s} = \left(2\alpha X_{s} -\frac{7}{4}X_{s}^{2}\right) ds + \sqrt{X_{s}}dB_{s},~~s>0,
    \quad
    \text{ and }
    \quad
    X_0 = \delta^{-1}.
\end{equation}
\end{lem}
\noindent
For more background on SDEs in general and the specific equation encountered here we refer to Section~\ref{sec:preliminaries}. Let us mention only that stochastic processes satisfying~\eqref{eq:SDEFellerDiffusionX} are well-studied and are also called in the literature \emph{logistic Feller diffusions} or \emph{Feller diffusions with logistic growth}. Generally, such processes satisfy an SDE of the form
\begin{equation}
\label{eq:generalSDEFellerDiffusionX}
	dX_{s} = (a X_{s} - c X_{s}^2) ds + \sqrt{\gamma X_{s}}dB_{s},\;s>0,\text{ with initial condition } X_0 = x \ge 0,
\end{equation}
where $a\in \mathbb{R}$ and $c,\gamma >0$. They appear in the context of population dynamics and stochastically extend the deterministic logistic growth model that describes the evolution of a population under the influences of natural birth, mortality and inter-individual competition. A prime source on the topic is Lambert \cite{lambert2005branching}, who provides a thorough and detailed discussion of the properties of solutions to \eqref{eq:generalSDEFellerDiffusionX}. 

With Lemma~\ref{thm:processconvergence} at hand we readily establish in Section~\ref{sec:mainProof}, see Lemma~\ref{lem:convergenceTprime} there, that the first time step at which the unhappy particles vanish, divided by~$n^{1/2}$, converges in distribution to the \emph{absorption time} of $X$, that is, the first time when $X$ hits zero. Letting $\delta \to 0$ then yields the claimed statement. In particular,~$T_\alpha$ in Theorem~\ref{thm:main} is the absorption time of the limiting solution of~\eqref{eq:SDEFellerDiffusionX} when the initial condition $X_0 \to \infty$; as we will see, this limiting process, called
\emph{standard logistic branching process}, is well-defined and well-studied, see~\cite{lambert2005branching,foucart2020entrance} and Section~\ref{ssec:LBprocess}.

The explicit descriptions of $X$ and $T_\alpha$ pave the way to obtain further bits of information. To achieve this we will exploit the following bounds, stating that $n^{-1/2}T_{n,M}$ has exponential tails, that are an immediate consequence of the main theorems in~\cite{DMP23}.
\begin{thm}
\label{cor:tailbounds}
Let $\alpha \in \mathbb{R}$ and $M = n/2 + \alpha n^{1/2} + o(n^{1/2}) \in \mathbb{N}$. Then there is a constant $c_\alpha >0 $ such that for all sufficiently large $n$ 
\begin{equation*}
    \mathbb{P}\big(T_{n,M}\le  n^{1/2}/Ac_\alpha\big)\le e^{-A}\quad \mbox{and}\quad \mathbb{P}\big(T_{n,M}>A c_\alpha n^{1/2}\big)\le e^{-A},
    \qquad
    A \ge 1.
\end{equation*} 
\end{thm}
\noindent
Together with Theorem~\ref{thm:main} we readily obtain convergence in ${\cal L}^p$, for any $p \in \mathbb{R}$, that is,
\begin{equation}
\label{eq:Lpconvergence}    
    n^{-p/2} \mathbb{E}\big[T_{n,M}^p\big]
    \sim \mathbb{E}\big[T_\alpha^p\big], \quad p \in \mathbb{R}.
\end{equation}
As we shall show in Lemma~\ref{seriesrepr}, we obtain for $M = n/2 + \alpha n^{1/2} + o(n^{1/2})$ the series representation
\begin{equation*}
    \mathbb{E}[T_\alpha]
    = \lim_{n\to\infty} \mathbb{E}\big[n^{-1/2}T_{n,M}\big]= \frac{\pi^{3/2}}{\sqrt{7}}+   \frac1{\sqrt{7}}\sum_{m \ge 1} \frac{\Gamma(\frac{m+1}{2})}{m!}\left(\frac{8\alpha}{\sqrt{7}}\right)^m t_m,
    \quad \alpha\in\mathbb{R},
\end{equation*}
where $\Gamma(\cdot)$ is the Gamma function and 
\[
    t_m
    \coloneqq \sum_{k \ge 0} \frac2{(\frac{m+1}{2} + 2k)(\frac{m+3}{2} + 2k)}
    = H_{(m-1)/4} - H_{(m-3)/4},
    \quad m \in \mathbb{N}_0,
\]
and $H_x = \sum_{k \ge 1} \big(\frac1k - \frac1{k+x}\big)$ denotes the `$x$-th harmonic number'. Let us highlight the specific case $\alpha = 0$: when we are essentially~\emph{at} the critical point, then we obtain the beautiful formula
\begin{equation}
\label{eq:ET0}    
    \mathbb{E}[T_0] = \lim_{n\to\infty} \mathbb{E}\big[n^{-1/2}T_{n,n/2 + o(n^{1/2})}\big] = \frac{\pi^{3/2}}{\sqrt{7}},
\end{equation}
which is in the interval $2.104 \pm 0.001$, see also Figure \ref{fig:threeRunsAtCrit}. Our methods also allow us to study the behavior of the transition in and out of the critical window, that is, $\mathbb{E}[T_\alpha]$ when $\alpha \to -\infty$ or $\alpha \to \infty$. In Section~\ref{sec:absorptiontime} we show that 
\[
    \mathbb{E}[T_\alpha] \stackrel{\alpha \to -\infty}{\sim} \frac{\ln|\alpha|}{|\alpha|}
    \qquad
    \text{and}
    \qquad
    \mathbb{E}[T_\alpha] \stackrel{\alpha \to \infty}{\sim} \frac{\sqrt{7\pi}}8\frac{e^{16 \alpha^2 / 7}}{\alpha^2}.    
\]
So, when $\alpha$ gets big, then $\mathbb{E}[T_\alpha]$ behaves (up to polynomial corrections) quadratic exponential in~$\alpha$; already for $\alpha = 3$ we obtain the enormous value $\mathbb{E}[T_3] \approx 5.894 \cdot 10^7$. On the other hand, for negative~$\alpha$ we get a moderate polynomial behavior with logarithmic corrections. Note that the large $|\alpha|$ asymptotics presented here are in accordance with the transition in and out of the critical window, see also Figure~\ref{fig:smoothtransition} and the discussion at the beginning of the introduction. To see this, note that if we set $M = n/2 + \alpha n^{1/2} + o(n^{1/2}) = (1+ \varepsilon)n/2$, then $\varepsilon \sim 2\alpha n^{-1/2}$, and so
\[
    \frac{\ln(\varepsilon^2 n)}{|\varepsilon|} \sim 2\frac{\ln|\alpha|}{|\alpha|} n^{1/2}
    \quad\text{and}\quad
    \varepsilon^{-1} e^{\Theta(\varepsilon^2n)}
    = \frac{n^{1/2}}{2\alpha}e^{\Theta(\alpha^2)}
    = \frac{e^{\Theta(\alpha^2)}}{\alpha^2}n^{1/2}.
\]
Our second main result addresses the total number of jumps $\sum_{t \ge 0} U_t$ performed by the particles until the dispersion time is reached.
\begin{thm}
\label{numbjumps}
Let $\alpha \in \mathbb{R}$ and $M = M(n) = n/2 + \alpha  n^{1/2} + o(n^{1/2}) \in \mathbb{N}$. Then there is a continuous random variable $A_\alpha$ such that, as $n\to \infty$,
\[
    n^{-1}\bigg(\sum_{t \ge 0} U_t - \frac27n\ln n\bigg)
    \overset{d}{\longrightarrow}
    A_\alpha.
\]
In particular, when $\alpha=0$ there exists a $\chi\in\mathbb{R}$ such that
\[
    \mathbb{P}(A_0 \ge a) = \mathrm{erf}\left(\frac{\sqrt7}2e^{-7(a-\chi)/4}\right), \quad a \in \mathbb{R},
\]
where $\mathrm{erf}(x) = \frac{2}{\sqrt{\pi}}\int_0^xe^{-y^2}dy, x \in \mathbb{R}$, denotes the error function.
\end{thm}
Regarding the definition and the actual value of $\chi$ we refer to Section~\ref{ssec:veryEarlySteps}, see also Remark~\ref{rem:chi}.
As a consequence of the theorem we deduce that each of the $M \sim n/2$ particles performs on average typically $\sim \frac47\ln n$ jumps before everybody settles, and this is independent of $\alpha$. However, the fluctuations in the total number of jumps are linear in $n$ and the limiting distribution depends on $\alpha$.
Towards the proof of Theorem~\ref{numbjumps} we show
in our analysis of the early, i.e., the first $o(n^{1/2})$, steps, which can be found in Section~\ref{sec:earlySteps}, that there are already $\sim \frac27 n\ln n$ jumps in those steps of the process. With Lemma~\ref{thm:processconvergence} and Theorem~\ref{thm:main} in mind, it is not surprising that the remaining $\Theta(n^{1/2})$ steps only contribute an additional $O(n)$ jumps, as $n^{-1/2}U_t$ is typically bounded for $t = \Theta(n^{1/2})$. We verify this intuition in the proof of Theorem~\ref{numbjumps} in Section~\ref{sec:proofnumjumps}. In particular, Theorem~\ref{thm:main} and Lemma~\ref{thm:processconvergence} assert that $(n^{-1/2}U_{T_{n,M,\delta} + \lfloor sn^{1/2} \rfloor})_{s \ge 0}$ converges weakly to a logistic branching process $X$, and so the total number of jumps behaves like $n^{1/2} A_{\alpha,\delta}$, where $A_{\alpha,\delta} \coloneqq \int_0^\infty X_s ds$, plus the additional $\frac27 n\ln n$ jumps from the first $T_{n,M,\delta}$ steps. Some care needs to be taken and some significant extra work needs to be done to take the $\delta \to 0$ limit; in particular, we need a sufficiently tight control of $U_t$ for $t \le T_{n,M,\delta}$, so that the errors do add up to $o(n)$, and this is accomplished by means of a delicate martingale argument in Section~\ref{sec:earlySteps}. All pieces are then put together in Section~\ref{sec:proofnumjumps}. 

\paragraph{Variations on the Theme.} Our work opens up opportunities for studying a variety of models that are related to the dispersion process or extensions of it. In a general setting, \emph{happiness} can be defined as a property of individual vertices and particles. More specifically, each vertex may have a \emph{capacity}, which, if exceeded, deems all particles on that vertex as unhappy. On the other side, each particle $p$ may have a \emph{stress level}, which dictates an upper bound on the particles that share a vertex  with $p$ so that $p$ is still happy. We leave it as an open problem to study the precise behavior in a general setting, where for example the empirical distributions of the capacities and the stress levels fulfill appropriate convergence properties.

In a different line of research it would be challenging to provide detailed studies of dispersion processes on graphs different than the complete graph. We believe, for example, that our results also hold if the underlying graph is a sufficiently dense Erd\H{o}s-R\'enyi random graph, which is obtained by retaining independently each edge of the complete graph on $n$ vertices with probability~$p$. In particular, if, say, $p = \omega(n^{-1/2})$, guaranteeing that the minimum degree is much larger than $n^{1/2}$, then similar results as in Theorem~\ref{thm:main} should hold, as the process finishes after $O(n^{1/2})$ rounds if the graph is complete. However, it might be the case that even on much sparser graphs the behavior does not change (since, for example, in most steps just an $O(n^{1/2})$  number of particles move). We consider it as an important and eminent challenge to study the effect of the edge probability $p$ on the distribution of the dispersion time. 

\paragraph{Related work.}
The aforementioned paper~\cite{FP18} studies, apart from complete graphs, the dispersion time on several other families, like paths, trees, grids, hypercubes and Cayley
graphs. The dispersion process was also studied by Frieze and Pegden~\cite{FP18}, who considered the \emph{dispersion distance} on the infinite line $L_{\infty}$. They showed that the dispersion distance is $\Theta(n)$ when there are~$n$ particles in the system, improving upon previous results in~\cite{CMRRS18}. A similar setup  was considered by Shang~\cite{S20}, who studied the dispersion distance on $L_{\infty}$ in a non-uniform dispersion process.

Processes where particles move on the vertices of a graph have been widely studied over the past decades; we refer to \cite{CMRRS18} for references. Concerning processes whose scope is to \textit{disperse} particles on discrete structures, arguably the best known such model is Internal Diffusion Limited Aggregation, see \cite{DF91,LBG92,10.1007/978-3-662-43948-7_21}. There, particles sequentially start (one at a time) from a specific vertex designated as the origin.
Each particle moves randomly until it finds an unoccupied vertex; then it occupies it forever and it does not move at subsequent steps.  
Another well-studied class of models are Activated Random Walks that evolve on $\mathbb{Z}^d$, see~\cite{ar:rollaARW} for an extensive review. Roughly speaking, we place particles on~$\mathbb{Z}^d$, and some of them are initially active while others are asleep. The rules of the process are then as follows. Whenever a particle is alone on a vertex, it falls asleep with a certain rate. On the other hand, active particles jump according to independent random choices, and whenever they encounter a particle that is asleep, they wake it up.

~\\
\noindent
In the remainder of the section we provide a formal definition of the model that we study, then we fix some notation, and eventually we conclude with a brief outline of the paper.

\paragraph{Model.} Let $n \in \mathbb{N}$. We denote by $M\ge 2$  the number of particles and we write $\mathcal{U}_t$ and $\mathcal{H}_t$ for the \textit{sets} of unhappy and happy particles at step $t$, respectively.
Moreover, we set $U_t\coloneqq |\mathcal{U}_t|$ and $H_t\coloneqq |\mathcal{H}_t|$. Initially, which is at step $t=0$, all $M$ particles are placed on \textit{one} distinguished vertex, say vertex 1, and are unhappy. Thus, writing $p_i$ for the $i$-th particle, $1\le i \le M$, we set
\[
    \mathcal{U}_0 = {\cal P}\coloneqq \{p_1, \dots, p_M\},~~ U_0 = M
    \quad \text{and} \quad 
    \mathcal{H}_0 = \emptyset,~~ H_0 = 0.
\]
For every $t\in \mathbb{N}_0$, the distribution of $\mathcal{H}_{t+1}$ (and thus also of $\mathcal{U}_{t+1}, U_{t+1},H_{t+1}$), given $\mathcal{U}_{t}$, is defined as follows.
Each particle in $\mathcal{U}_{t}$ moves to one of the $n$ vertices selected independently and uniformly at random and each particle in $\mathcal{H}_t$ remains at its position. In particular, if we denote by $p_{i,t}$ the position of particle $i$ at step $t\in \mathbb{N}_0$, then
\[
    p_{i,0} = 1, ~ 1\le i \le M,
\quad
\text{and, in distribution,}
\quad
    p_{i,t+1} =
    \begin{cases}
        p_{i,t}, & \text{ if } p_i \in {\cal H}_t, \\
        G_{i,t+1}, & \text{ if } p_i \in {\cal U}_t, 
    \end{cases},
    ~ 1 \le i \le M, t \in \mathbb{N}_0,
\]
where $(G_{i,t})_{1 \le i \le n, t\in\mathbb{N}_0}$ are independent and uniform from $\{1, \dots, n\}$. In addition, we set for $t \in \mathbb{N}_0$
\[
    {\cal H}_{t+1} = \big\{ p_i \in {\cal P}: p_{i,t+1} \neq p_{j,t+1} \text{ for all $j \in \{1, \dots, M\} \setminus \{i\}$}\big\},
    \quad
    {\cal U}_{t+1} = {\cal P} \setminus {\cal H}_{t+1}.
\]
With this notation, the \emph{dispersion time} considered in Theorem~\ref{thm:main}
is defined as the smallest~$t$ at which there are no unhappy particles, that is,
\begin{equation}
\label{eq:TnM}
    T_{n, M} \coloneqq \inf\big\{t \in \mathbb{N}_0 : U_t = 0\big\}.
\end{equation}

\paragraph{Notation.} Let $\mathbb{N}$ denote the set of positive integers and set $\mathbb{N}_0=\mathbb{N}\cup \{0\}$. Given $k\in \mathbb{N}$, we write $[k]\coloneqq \{1,\dots,k\}$. 
Given functions $f:\mathbb{N}\mapsto \mathbb{R},g:\mathbb{N}\mapsto \mathbb{R}$, we write $f=o(g)$ if $|f(n)/g(n)|\rightarrow 0$ as $n\rightarrow \infty$ and $f=O(g)$ if there is a constant $C>0$ such that $|f(n)|\leq C|g(n)|$ for all large enough $n$. We write $f=\Theta(g)$ if $f=O(g)$ and $g=O(f)$, whereas the notation $f=\omega(g)$ means that $|f(n)/g(n)|\rightarrow \infty$ as $n\rightarrow \infty$. For convenience,  given functions $f,g,h:\mathbb{R}\to\mathbb{R}$ we will write
\[
    f(u) = g(u) \pm h(u)  \Longleftrightarrow |f(u)-g(u)|\leq h(u), \quad u \in \mathbb{R}.
\]
Given $a,b,c\in\mathbb{R}$, we set $a \vee b \coloneqq \max\{a,b\}$, $a \wedge b \coloneqq \min\{a,b\}$. Moreover, for $b,c\neq0$ we write, whenever it is not ambiguous,  $a/bc$ (instead of, say, $a/(bc)$) for $\frac{a}{bc}$ and $ab/c$ for $\frac{ab}{c}$. 

\paragraph{Outline.} The paper is structured as follows. In the next section we give some basic background on SDEs and present the main tool, diffusion approximation, that we use to study the process of the number of unhappy particles. Moreover, we take a closer look at the logistic branching processes and collect the properties that will be relevant here; our main source is the thorough study~\cite{lambert2005branching}. Moreover, we collect some facts about martingales that will be useful.
In order to apply the diffusion approximation framework it is necessary to study the \emph{drift} (expected change) and the \emph{variation} (square of the expected change) of the underlying Markov chain; this is performed in Section~\ref{sec:driftvar} for the process of unhappy particles.
Subsequently, in Section~\ref{sec:earlySteps} we study the early steps of the dispersion process and provide, by means of martingale concentration arguments, tight bounds for $U_t$ for $t = o(n^{1/2})$.
With this at hand we prove Theorems~\ref{thm:processconvergence} and ~\ref{thm:main} in Section~\ref{sec:mainProof}. In Section~\ref{sec:absorptiontime} we study the expectation of the absorption time and prove the explicit and asymptotic formulae. Finally, in Section~\ref{sec:proofnumjumps} we prove Theorem~\ref{numbjumps}  about the total number of jumps.

\section{Probabilistic Preliminaries}
\label{sec:preliminaries}

In this section we collect several probabilistic facts that we will use: diffusion approximation, properties of the logistic Feller diffusion and martingale inequalities. 

\subsection{Diffusion Approximation}
\label{ssec:preliminariesdiffusion}

A main tool that we will use in the proof of Theorem~\ref{thm:main} is the concept of \emph{diffusion approximation}, which allows us to approximate a sequence $(\mathbf{Y}^{(n)})_{n\in\mathbb{N}}$ of Markov chains with values in $\mathbb{R}^d$, where ${d \in \mathbb{N}}$, by a continuous-time stochastic process. More specifically, we examine convergence properties of~$(\mathbf{Y}^{(n)})_{n\in\mathbb{N}}$ to a process satisfying a stochastic differential equation (SDE)
\begin{equation}
\label{eq:SDE}
    d\mathbf{X}_{s} = b(\mathbf{X}_{s})ds + \sigma(\mathbf{X}_{s})d\mathbf{B}_{s},
    \quad s>0,
\end{equation}
where $b: \mathbb{R}^{d}\to\mathbb{R}^{d}$ and $\sigma: \mathbb{R}^{d} \to \mathbb{R}^{d\times d}$ are suitable functions and $\mathbf{B}$ is a $d$-dimensional standard Brownian motion.
In this section we provide an overview of the necessary results from stochastic calculus.
Additionally, we collect some properties of the limit process that will emerge within the proof of Theorem~\ref{thm:main}.
In what follows we denote discrete time by $t \in \mathbb{N}_{0}$ (so, for example, $\mathbf{Y}^{(n)} = (\mathbf{Y}^{(n)}_t)_{t\in\mathbb{N}_0}$), whereas~$s\ge 0$ represents continuous time. Moreover, for all $i,j\in [d]$, the subscript $i$ denotes the $i$-th component of a $d$-dimensional vector and the subscript $ij$ refers to the entry in row $i$ and column $j$ of a $d\times d$-dimensional matrix. 

Let us consider \eqref{eq:SDE}.
A \emph{(weak) solution to \eqref{eq:SDE} with initial value $\mathbf{X}_0 = \mathbf{x} \in \mathbb{R}^{d}$} is a triple $(\mathbf{X},\mathbf{B},\mathscr{P})$, where $\mathscr{P} = (\Omega, \mathcal{F}, (\mathcal{F}_{s})_{s\geq 0},\mathbb{P})$ is a filtered probability space with the filtration satisfying the usual conditions, i.e.~$(\mathcal{F}_{s})_{s\geq0}$ is right-continuous and complete. Further, $\mathbf{X}=(\mathbf{X}_{s})_{s\geq0}$ and $\mathbf{B}=(\mathbf{B}_{s})_{s\geq0}$ are continuous stochastic processes that are adapted to  $(\mathcal{F}_{s})_{s\geq0}$ such that 
\begin{itemize}
\item $\mathbf{B}$ is a standard $d$-dimensional Brownian motion with respect to $(\mathcal{F}_{s})_{s\geq0}$, i.e.~$\mathbf{B}$ is a standard Brownian motion and $\mathbf{B}_s-\mathbf{B}_r$ is independent of $\mathcal{F}_r$ for any $0\leq r<s$;
\item $\mathbf{X}_s$ satisfies \eqref{eq:SDE} and the initial condition, i.e.
\[
    \mathbf{X}_{s}
    = \mathbf{x}  + \int_{0}^{s}b(\mathbf{X}_{r})dr + \int_{0}^{s}\sigma(\mathbf{X}_{r})d\mathbf{B}_{r},
    \quad s\ge 0,
\]
or, equivalently, if we write $\mathbf{X}_s = (X_{1,s},\dots,X_{d,s})$ for $s \ge 0$, then
\[
    X_{i,s} = x_{i} + \int_{0}^{s}b_{i}(\mathbf{X}_{r})dr + \sum_{j=1}^{d}\int_{0}^{s}\sigma_{ij}(\mathbf{X}_{r})dB_{j,r}, \quad i \in [d], ~ s \ge 0.
\]    
\end{itemize}
Moreover, we say that there is \emph{(weak) uniqueness} if whenever $(\mathbf{X},\mathbf{B},\mathscr{P})$ and $(\mathbf{\tilde{X}},\mathbf{\tilde{B}},\tilde{\mathscr{P}})$ solve \eqref{eq:SDE} weakly and satisfy $\mathbf{X}_{0}=\mathbf{\tilde{X}}_{0}$, then $\mathbf{X}$ and $\mathbf{\tilde{X}}$ have the same law.

In order to get the diffusion approximation to work, we construct a sequence of right-continuous and continuous-time stochastic processes from the given sequence  $(\mathbf{Y}^{(n)})_{n\in\mathbb{N}}$ of discrete time Markov chains by using constant interpolation between the time points. Then, under appropriate conditions specified in the subsequent theorem, $(\mathbf{Y}^{(n)})_{n\in\mathbb{N}}$ converges weakly
to the solution of an SDE.
With the necessary  concepts at hand we are now ready to present our main tool, and we refer for example to~\cite[Ch.~8]{durrett2018stochastic} for an extensive treatment.
\begin{thm}[Diffusion Approximation]
\label{thm:diffusionapproximation}
Let $d \in\mathbb{N}$, $\sigma: \mathbb{R}^{d} \to \mathbb{R}^{d \times d}$ and $b: \mathbb{R}^{d} \to \mathbb{R}^{d}$ be continuous functions and assume that for any $\mathbf{x} \in\mathbb{R}^{d}$ the SDE~\eqref{eq:SDE} possesses a unique solution such that $\mathbf{X}_{0}=\mathbf{x} $. Furthermore, let $h:\mathbb{N} \to \mathbb{R}_{+}$ be a sequence with $\lim_{n\to\infty}h(n) = 0$ and for all~$n \in \mathbb{N}$ let $\mathbf{Y}^{(n)} = (\mathbf{Y}_{t}^{(n)})_{t\in\mathbb{N}_{0}}$ be a discrete-time Markov chain with values in~$S^{(n)}\subseteq\mathbb{R}^{d}$. Define, for all $t\in\mathbb{N}_{0}$, $\mathbf{x} \in S^{(n)}$ and $i,j \in [d]$
\begin{equation*}
    b_{i}^{(n)}(\mathbf{x} ) \coloneqq \frac{\mathbb{E}\big[Y_{i,t+1}^{(n)} - x_{i} \mid \mathbf{Y}_{t}^{(n)} = \mathbf{x} \big]}{h(n)},
    \quad
    a_{ij}^{(n)}(\mathbf{x} ) \coloneqq  \frac{\mathbb{E}\big[(Y_{i,t+1}^{(n)} - x_{i})(Y_{j,t+1}^{(n)} - x_{j})\mid \mathbf{Y}_{t}^{(n)} = \mathbf{x} \big]}{h(n)},
\end{equation*}
and $\gamma^{(n)}_{p}(\mathbf{x} ) \coloneqq  \mathbb{E}\big[|\mathbf{Y}^{(n)}_{t+1} - \mathbf{x} |^{p}\mid \mathbf{Y}_{t}^{(n)} = \mathbf{x} \big] / h(n)$ for $p\geq2$.
Let $a\coloneqq \sigma \sigma^{T}$ and assume that for all $R<\infty$ and $i,j\in [d]$
\begin{equation*}
   \lim_{n\to\infty} \sup_{\mathbf{x} \in S^{(n)},|\mathbf{x} |\leq R} |b_{i}^{(n)}(\mathbf{x} ) - b_{i}(\mathbf{x} )|= 0, \quad 
   \lim_{n\to\infty} \sup_{\mathbf{x} \in S^{(n)},|\mathbf{x} |\leq R} |a_{ij}^{(n)}(\mathbf{x} ) - a_{ij}(\mathbf{x} )|= 0,
\end{equation*}
and 
\begin{equation*}
    \lim_{n\to\infty} \sup_{\mathbf{x} \in S^{(n)},|\mathbf{x} |\leq R} \gamma^{(n)}_{p}(\mathbf{x} ) = 0 ~\text{ for some }~ p \geq 2.
\end{equation*}
Finally, assume that $\mathbf{Y}^{(n)}_{0} \to \mathbf{x} $ as $n \to \infty$.
Then $(\mathbf{Y}^{(n)}_{\lfloor s/h(n) \rfloor})_{s\geq0}$ converges weakly to a strong Markov process~$\mathbf{X}$ that satisfies the SDE~\eqref{eq:SDE} with $\mathbf{X}_{0} =\mathbf{x}$. 
\end{thm}

\subsection{The (Standard) Logistic Branching Process}
\label{ssec:LBprocess}

We already discussed in the introduction that the processes that will be relevant here are the so-called logistic branching processes, given by the solution of 
\[
	dX_{s} = (a X_{s} - c X_{s}^2) ds + \sqrt{\gamma X_{s}}dB_{s},\;s>0,
\]
with $X_{0} = x \ge 0$, $a\in \mathbb{R}$ and $c,\gamma >0$,  see also~\eqref{eq:generalSDEFellerDiffusionX} (and~\eqref{eq:SDEFellerDiffusionX} for the particular case that will appear here). 
In the remainder of this section we collect some key properties that will be handy. The first one is about the existence and uniqueness of solutions, see for example \cite[Cor.~2.11]{foucart2020entrance}, ~\cite[Sec.~3]{lambert2005branching}, \cite[Thm.~5.1]{fu2010stochastic} and references therein.
\begin{lem}
\label{lem:existenceanduniqueness}
For all initial states $x \ge 0$ and for all $a \in \mathbb{R}$ and $c,\gamma >0$, there exists a unique solution $(X_{a,c,\gamma,x},B_{a,c,\gamma,x},\mathscr{P}_{a,c,\gamma,x})$ to \eqref{eq:generalSDEFellerDiffusionX}. Moreover,~$X_{a,c,\gamma,x}$ is non-negative.
\end{lem}
In what follows it will be convenient to consider a specific choice of the filtered probability space $\mathscr{P}_{a,c,\gamma,x} = (\Omega', \mathcal{F}', (\mathcal{F}'_s)_{s \ge 0}, \mathbb{P}')$ (where all components depend on the parameters $a,c,\gamma,x$) from the previous lemma that we construct as follows.
Let $C([0,\infty),\mathbb{R})$ be the space of all continuous maps $\xi:[0,\infty) \to\mathbb{R}$ and let $X$ be the coordinate process given by $X_s(\xi) = \xi(s)$ for all $s\geq0$ and ${\xi \in C([0,\infty),\mathbb{R})}$.
Additionally, consider the $\sigma$-algebra $\mathcal{F}=\sigma\{X_s \mid s\ge 0\}$ and equip the measurable space $(C([0,\infty),\mathbb{R}), \mathcal{F})$ with the filtration $(\mathcal{F}_s)_{s\ge 0}$ given by ${\mathcal{F}_s=\sigma\{X_r \mid 0\le r\le s\}}$ for all $s\geq0$, which we may complete and right-continuously extend in order to fulfil the usual conditions. 
Since $X_{a,c,\gamma,x}$ is a continuous process, it is possible to switch from $\mathscr{P}_{a,c,\gamma,x}$ to the canonical probability space $(C([0,\infty),\mathbb{R}),\mathcal{F},(\mathcal{F}_{s})_{s\ge 0},\mathbb{P}_{a,c,\gamma,x})$ via the map ${\Omega' \ni \xi' \mapsto X_{a,c,\gamma,x}(\xi') \in C([0,\infty),\mathbb{R})}$, where $\mathbb{P}_{a,c,\gamma,x}$ is the probability measure given by $\mathbb{P}_{a,c,\gamma,x}(A)= {\mathbb{P}'}((X_{a,c,\gamma,x})^{-1}(A))$ for all $A\in\mathcal{F}$. By this particular choice of probability measure we obtain that the coordinate process~$X$ on the space $(C([0,\infty),\mathbb{R}),\mathcal{F},(\mathcal{F}_{s})_{s\ge 0},\mathbb{P}_{a,c,\gamma,x})$ has the same law as $X_{a,c,\gamma,x}$ under $\mathbb{P}'$, i.e.~under $\mathbb{P}_{a,c,\gamma,x}$ the process~$X$ satisfies \eqref{eq:generalSDEFellerDiffusionX}. The following corollary is now an immediate consequence of Lemma~\ref{lem:existenceanduniqueness}, and a similar construction was also performed in~\cite{lambert2005branching}.
\begin{cor}
\label{cor:existenceanduniqueness}
For all initial states $x\ge 0$ and for all $a\in\mathbb{R}$ and $c,\gamma>0$, there is a unique solution $(X,B_{a,c,\gamma,x},\mathscr{P}_{a,c,\gamma,x})$ to \eqref{eq:generalSDEFellerDiffusionX}, where $X$ is the coordinate process and thus independent of $a,c,\gamma,x$. Moreover, $X$ is non-negative $\mathbb{P}_{a,c,\gamma,x}$-almost surely, where $\mathbb{P}_{a,c,\gamma,x}$ denotes the probability measure of $\mathscr{P}_{a,c,\gamma,x}$.
\end{cor}
For the rest of this paper we will adopt the above procedure and consider solutions to~\eqref{eq:generalSDEFellerDiffusionX} only with respect to the canonical probability space $(C([0,\infty),\mathbb{R}),\mathcal{F},(\mathcal{F}_{s})_{s\geq0},\mathbb{P}_{a,c,\gamma,x})$.
Our first main object of interest in Theorem~\ref{thm:main} is the time at which the logistic branching process $X$ hits zero. To this end, define
\[
    T(\xi) \coloneqq \inf\big\{s\ge 0 : \xi(s)=0\big\}, 
    \quad
    \xi \in C([0,\infty),\mathbb{R}).
\]
The author of \cite{lambert2005branching} establishes that $T$ is finite $\mathbb{P}_{a,c,\gamma,x}$-almost surely. Moreover, $X_{s}=0$ for all ${s\geq T}$ under $\mathbb{P}_{a,c,\gamma,x}$, i.e.~upon hitting zero the process becomes constant, which is why we also refer to~$T$ as \emph{absorption time}.
Although it seems to be well-known, we could not find an explicit reference for the fact that $T$ is continuous if the underlying process is a logistic Feller diffusion.
\begin{lem}
    \label{lem:continuityabsorptiontime}
    Let $x>0, a \in \mathbb{R}$ and $c,\gamma >0$. Then $T$ is continuous under $\mathbb{P}_{a,c,\gamma,x}$.
\end{lem}
\begin{proof}
    Since $X$ has continuous paths and $X_{0}=x>0$ under $\mathbb{P}_{a,c,\gamma,x}$, we obtain $\mathbb{P}_{a,c,\gamma,x}(T=0)=0$. The continuity of $T$ then follows for example from~\cite[Thm.~5.1]{kent1980eigenvalue}, which states that $T$ is an infinite convolution of elementary mixtures of exponential distributions.
\end{proof}
Within our context, it will be necessary to consider solutions to \eqref{eq:generalSDEFellerDiffusionX} with initial value $x \to \infty$. The required results are covered by the following statement, whose proof can be found in \cite[Thm.~3.9, Cor.~3.10]{lambert2005branching}. Set 
\begin{equation}
    \label{eq:theta}
    \theta:[0,\infty) \to \mathbb{R},
    \qquad
    \lambda \mapsto \int_{0}^{\lambda}\exp\left(\frac{\gamma}{4c}v^{2} - \frac{a}{c}v\right)dv.
\end{equation}
\begin{lem}
\label{lem:infiniteinitialvalue}
Let $x\geq0, a \in \mathbb{R}$ and $c,\gamma >0$.
Then the expectation of $T$ under $\mathbb{P}_{a,c,\gamma,x}$ is finite and 
\[
    \mathbb{E}_{a,c,\gamma,x}\left[T\right] 
    = \frac1c\int_0^\infty\frac{\theta(\lambda)}{\lambda \theta'(\lambda)}(1-e^{-x\lambda})\;d\lambda.
\]
In addition, the measures $(\mathbb{P}_{a,c,\gamma,x})_{x\geq0}$ converge weakly, as $x\to\infty$, to the law $\mathbb{P}_{a,c,\gamma,\infty}$ of the so-called \emph{standard logistic branching process}.
Under $\mathbb{P}_{a,c,\gamma,\infty}$, the absorption time
$T$ is a continuous random variable that is finite almost surely and has finite expectation given by
\[ \mathbb{E}_{a,c,\gamma,\infty}\left[T\right] = \sup_{x\ge 0}\mathbb{E}_{a,c,\gamma,x}\left[T\right] =  \frac{1}{c}\int_{0}^{\infty}\frac{\theta(\lambda)}{\lambda \theta'(\lambda)}\;d\lambda. \]
\end{lem}
We remark that the continuity of $T$ under $\mathbb{P}_{a,c,\gamma,\infty}$ is not explicitly claimed in~\cite{lambert2005branching}. However, it follows from Lemma~\ref{lem:continuityabsorptiontime}, as we can apply the Markov property of $X$ and restart the process at some $x>0$ to express $T$ under $\mathbb{P}_{a,c,\gamma,\infty}$ as the sum of two independent random variables that are distributed as $T$ under $\mathbb{P}_{a,c,\gamma,x}$ and the hitting time $\inf\{s\ge0 \mid \xi(s)=x\}$ under $\mathbb{P}_{a,c,\gamma,\infty}$.

Within the proof of Theorem~\ref{numbjumps} we will also need information about the distribution of $\int_{0}^{\infty}X_{s}ds$, where~$X$ is a logistic branching process that satisfies \eqref{eq:generalSDEFellerDiffusionX}.
Set
\[
    A_s(\xi) \coloneqq \int_0^s\xi(u)du, \quad s \ge 0, ~\xi \in C([0,\infty),\mathbb{R}) ~ .
\]
In order to study the distribution of $A_{\infty}$ under $\mathbb{P}_{a,c,\gamma,x}$ the following lemma from~\cite[Prop.~3.1]{lambert2005branching} will be very handy, as it establishes that under an appropriate time-change the logistic branching process fulfills another SDE, which, in turn, is much easier to analyse.
\begin{lem}
\label{lem:timechangedOUprocess}
For any $\xi \in C([0,\infty),\mathbb{R})$ let $A(\xi) \coloneqq (A_s(\xi))_{s \ge 0}$ and let
$A^{-1}(\xi)$ be its right-inverse, i.e. $(A(\xi) \circ A^{-1}(\xi))_s = s$ for all $s \in \{A_u(\xi)\mid u\ge0\}$.
Let $X$ be the coordinate process. Then, for all $x\ge0$, $a\in \mathbb{R}$ and $c,\gamma>0$, the process $R \coloneqq X \circ A^{-1}$ under $\mathbb{P}_{a,c,\gamma,x}$ satisfies the SDE
\begin{equation}
\label{eq:SDEOUprocess}
    dR_{s} = (a - cR_{s})ds + \sqrt{\gamma} dB_{s},\; s \in [0,A_{\infty}], \text{ with initial condition } R_{0}=x.
\end{equation}
\end{lem}
Processes satisfying~\eqref{eq:SDEOUprocess} for all $s \ge 0$ are called \emph{Ornstein-Uhlenbeck (O-U) processes} and are very well-studied. With this lemma at hand we show the following handy relation between $A_\infty$ (under $\mathbb{P}_{a,c,\gamma,x}$) and the first time that the corresponding O-U process hits zero.

\begin{lem}
\label{lem:integralequalsTprime}
Let $R$ be defined as in Lemma~\ref{lem:timechangedOUprocess}. For all $x\ge 0$, $a\in \mathbb{R}$ and $c,\gamma>0$, 
\[A_\infty = T(R) = \inf\{s \ge 0: R_s = 0\} \quad \mathbb{P}_{a,c,\gamma,x}\text{-almost surely.}\]
\end{lem}

\begin{proof}
Let $X,A^{-1}$ be as in Lemma~\ref{lem:timechangedOUprocess}.
Let $\xi \in C([0,\infty),\mathbb{R})$ be non-negative such that $\xi(s)=0$ for all $s\ge T(\xi)$, so that $\xi(s) > 0$ for all $s\in[0,T(\xi))$. In order to simplify notation we omit the dependence on $\xi$ in the forthcoming calculations, that is, we write $R_s$ for $R_s(\xi)$, $R = (R_s)_{s \ge 0}$ and similarly for $A,A^{-1},X$ and $T$; in particular, since $X$ is the coordinate process, $X = \xi$. Note that by our choice of $\xi$ the integral $A$ is strictly increasing on $[0,T)$ and hence invertible on $[0,T)$, i.e., $A^{-1}$ is also the left-inverse of $A$ on $[0,T)$. Therefore, $A^{-1}(A_{T})\ge T$ and, by definition of $R$ and since $\xi(s) = 0$ for all $s\ge T$,
\[
    R_{A_T} = (X \circ A^{-1})_{A_T} = \xi(A^{-1}(A_{T})) = 0,
\]
so that $A_T \ge T(R)$.
In order to see the converse inequality, note that, again by definition of $R$,
\[
    0
    = R_{T(R)}
    = X_{A^{-1}_{T(R)}}
    = \xi(A^{-1}_{T(R)}),
\]
and so $T \le A^{-1}_{T(R)}$. Since $\xi$ is non-negative, $A$ is non-decreasing and so
\begin{equation*}
     A_T \le A_{A^{-1}_{T(R)}} = T(R).
\end{equation*}
All in all, $A_T = T(R)$ whenever $\xi$ is non-negative with $\xi(s)=0$ for all $s\ge T$. To conclude the proof, Corollary~\ref{cor:existenceanduniqueness} guarantees that $\mathbb{P}_{a,c,\gamma,x}$-almost surely, the process $X$ is non-negative and absorbing, where the latter implies~$A_T = A_\infty$.
\end{proof}
The connection established in the previous lemma will be very useful, as the hitting times of O-U processes are well-studied, see for example~\cite[Rem.~2.3]{alili2005representations}. We just need the case $a=0$.
\begin{lem}
\label{lem:densityTprime}
Let $R$ be defined as in Lemma~\ref{lem:timechangedOUprocess}. For all $x\ge 0$ and $c,\gamma>0$, the density of $T(R)$ under $\mathbb{P}_{0,c,\gamma,x}$ is given by 
\[
    p(s) = \frac{2c^{3/2} x}{\sqrt{\gamma \pi}}e^{2cs}\left(e^{2cs} - 1\right)^{-3/2}\exp\left(-\frac{cx^2}{\gamma(e^{2cs}-1)} \right),\quad s\ge0. 
\]
\end{lem}


\subsection{Estimates for Martingales (with Subgaussian Increments)}
\label{ssec:preliminariesmartingales}

Here we collect some preliminary facts concerning subgaussian random variables and martingales. Following standard notation, see for example the extensive book~\cite{Boucheron2004}, we call a random variable~$X$ \emph{subgaussian with parameter} $\sigma^2  >0$ if 
\[
    \mathbb{E}\Big[e^{s(X-\mathbb{E}[X])}\Big]
    \le e^{s^2\sigma^2/2} \text{ for every } s\in \mathbb{R}.
\]
The following result provides a handy way of checking whether a random variable is subgaussian.
\begin{thm}[Thm.~2.1 in~\cite{Boucheron2004}]
\label{subG}
Let $X$ be a random variable with $\mathbb{E}[X]=0$. If, for some $v>0$, 
\[
    \mathbb{P}(X>\lambda),\mathbb{P}(X<-\lambda)\leq e^{-\lambda^2/2v}
    \quad
    \text{for all } \lambda >0,
\]
then $X$ is subgaussian with parameter $\sigma^2=16v$.
\end{thm}
Let $(\mathcal{F}_t)_{t\in \mathbb{N}_0}$ be a filtration on some probability space. Let $(M_t)_{t\in \mathbb{N}_0}$ be a real-valued, discrete-time stochastic process adapted to $(\mathcal{F}_t)_{t\in \mathbb{N}_0}$ and such that $\mathbb{E}[M_t]<\infty$ for each $t$. We say that $(M_t)_{t\in \mathbb{N}_0}$ is called a \textit{martingale} if $\mathbb{E}[M_{t+1}|\mathcal{F}_t]=M_t$ for each $t\in \mathbb{N}_0$. We say that $(M_t)_{t\in \mathbb{N}_0}$ is a \emph{submartingale} if $\mathbb{E}[M_{t+1}|\mathcal{F}_t]\geq M_t$, whereas it is called a \emph{supermartingale} if $\mathbb{E}[M_{t+1}|\mathcal{F}_t]\leq M_t$ for $t \in \mathbb{N}_0$. We continue with a simple lemma that provides a concentration bound for subgaussian supermartingales. Even though such a result is probably known or implicit in other works, we could not find a proper reference and hence we include the (short) proof for completeness.
\begin{lem}\label{lem:martingale}
Let $(M_t)_{t\in \mathbb{N}_0}$ be a supermartingale with respect to a filtration $(\mathcal{F}_{t})_{t\in \mathbb{N}_0}$ such that $M_0=0$ and $M_t$ conditional on $\mathcal{F}_{t-1}$ is a subgaussian random variable with parameter~$\sigma_t^2$. Then
\[
    \mathbb{P}(M_t\ge \lambda)
    \le
    \exp\left(-\frac{\lambda^2}{2\sum_{i=1}^t \sigma_i^2}\right),
    \qquad
     \lambda>0,~~t\in \mathbb{N}_0 \enspace.
\]
\end{lem}
\begin{proof}
Let $s \ge 0$.
Using the supermartingale property of $M_t$ and the fact that $M_t$, conditional on ${\cal F}_{t-1}$, is subgaussian with parameter $\sigma_t^2$, we readily obtain that
$$
    \mathbb{E}\big[e^{s (M_t-M_{t-1})}\mid \mathcal{F}_{t-1}\big]
    \le \mathbb{E}\big[e^{s (M_t-\mathbb{E}[M_{t}\mid \mathcal{F}_{t-1}])}\mid \mathcal{F}_{t-1}\big]
    \le e^{s^2 \sigma_t^2/2}.
$$
Using the tower property of conditional expectation and induction we see that 
\[
    \mathbb{E}\big[e^{s M_t}\big]
    = \mathbb{E}\big[e^{s M_{t-1}}\mathbb{E}\big[e^{s(M_t-M_{t-1})}\mid \mathcal{F}_{t-1}\big]\big]
    \le e^{s^2 \sigma_t^2/2}\mathbb{E}\big[e^{s M_{t-1}}\big]
    \le \exp\left(\frac12\sum_{i=1}^{t}s^2\sigma_i^2\right).
\]
Then, by Markov's inequality, 
\[
    \mathbb{P}(M_t\ge \lambda)
    = \mathbb{P}\big(e^{s M_t}\ge e^{s \lambda}\big)
    \le e^{-s \lambda} \mathbb{E}\big[e^{s M_t}\big]
    \le \exp\left(\frac12\sum_{i=1}^{t}s^2\sigma_i^2 - s \lambda\right).
\]
Choosing $s = \lambda / \sum_{1 \le i \le t}\sigma_i^2 > 0$
completes the proof.
\end{proof}
\noindent
Hoeffding's lemma implies that every bounded random variable is subgaussian and one can deduce the well-known Azuma-Hoeffding inequality, see e.g.~ \cite[Theorem 3.2.1]{roch_mdp_2024}.
\begin{thm}
\label{azumahoeffding}
Let $(M_t)_{t\in \mathbb{N}_0}$ be a supermartingale and let $N\in \mathbb{N}$. Suppose that $|M_i-M_{i-1}|\leq c_i$ for all $1\leq i\leq N$. Then
\[
    \mathbb{P}(M_N-M_0\ge b)\leq \exp\left(-\frac{b^2}{2\sum_{i=1}^{N}c^2_i}\right),
    \quad
    b \ge 0 \enspace .
\]
\end{thm}
\noindent


\section{Drift and Variation of the Unhappy Particles}
\label{sec:driftvar}

In this section we study the \textit{drift} and the \emph{variation} of the process of unhappy particles; more precisely, we determine exact and asymptotic expressions for  the quantities
\[
    \mathbb{E}\big[U_{t+1} - U_t \mid  U_t\big]
    \quad
    \text{and}
    \quad
    \mathbb{E}\big[(U_{t+1} - U_t)^2 \mid  U_t\big].    
\]
We begin as in \cite{DMP23} by writing
\begin{equation}
\label{eq:diffUt}
    U_{t+1} - U_t =  X_{t+1}-Y_{t+1},
    \quad
    \text{where}
    \quad
    X_{t+1}\coloneqq \left|\mathcal{H}_{t} \cap \mathcal{U}_{t+1}\right|,~
    Y_{t+1}\coloneqq \left|\mathcal{U}_{t}\cap \mathcal{H}_{t+1}\right|.
\end{equation}
That is, $X_{t+1}$ stands for the number of particles that were happy at step $t$ but become unhappy in step $t+1$ (because some unhappy particle in $\mathcal{U}_t$ moved onto their vertex) and $Y_{t+1}$ is the number of unhappy particles at time $t$ that become happy at step $t+1$ (because at time $t+1$ they are alone on the vertex that they occupy).
From here on we deviate from~\cite{DMP23}. Define
\[
    X_{t+1,h} \coloneqq \mathbb{1}[h \in \mathcal{U}_{t+1}]
    \quad\text{and}\quad
    Y_{t+1,u} \coloneqq \mathbb{1}[u \in \mathcal{H}_{t+1}]
\]
so that we can write
\begin{equation}\label{sums}
    X_{t+1} = \sum_{h \in \mathcal{H}_t} X_{t+1,h} \enspace
    \quad \text{and} \quad
    Y_{t+1} = \sum_{u \in \mathcal{U}_t} Y_{t+1,u}.
\end{equation}
The following simple lemma, also established in~\cite{DMP23}, determines the drift. We include a (short) proof here that uses our notation and because it is instructive for the computation of the second moment that will follow.
\begin{lem}\label{firstmomentexact}
Let $t\in \mathbb{N}_0$. Then 
\begin{equation*}\label{eq:FIRSTM}
\mathbb{E}\big[U_{t+1} - U_t\mid  U_t\big]
=  H_t\left(1-\left(1-\frac{1}{n}\right)^{U_t}\right) - U_t\frac{n-H_t}{n}\left(1-\frac{1}{n}\right)^{U_t-1}.
\end{equation*}
\end{lem}
\begin{proof}
The probability that none of the $U_t$ unhappy particles jump to a specific fixed vertex is $(1-1/n)^{U_t}$. So, an arbitrary happy particle $h \in {\cal H}_t$ is not joined by any unhappy particle (and thus remains happy) with that probability, and we obtain
\begin{equation}
\label{eq:EXh}
    \mathbb{E}\big[X_{t+1,h} \mid  U_t\big]
    = 1-\left(1-\frac{1}{n}\right)^{U_t},
    \quad
    h \in \mathcal{H}_t.
\end{equation}
Similarly, any unhappy particle $u\in {\cal U}_t$ has $n-H_t$ choices for a position $p_{u,t+1}$ that is not occupied by a happy particle, and all other $U_t-1$ unhappy particles will \textit{not} choose $p_{u,t+1}$ to move to with probability $(1-1/n)^{U_t-1}$. Thus 
\begin{equation}
\label{eq:EXh2}
    \mathbb{E}\big[Y_{t+1,u} \mid  U_t\big]
    = \frac{n-H_t}{n}\left(1-\frac{1}{n}\right)^{U_t-1}, \quad
    u \in \mathcal{U}_t.
\end{equation}
The statement follows from \eqref{eq:diffUt}, \eqref{sums} and linearity of expectation.
\end{proof}
We will need later the following simple, non-asymptotic estimate that was shown in~\cite{DMP23} and that follows rather easily from Lemma~\ref{firstmomentexact}. 
\begin{lem}
\label{lem:driftcoarsebounds}
Let $\varepsilon:\mathbb{N}\to [-1,1]$ and $M=M(n) \coloneqq(1+\varepsilon)n/2 \in \mathbb{N}$. Then, for any $t\in\mathbb{N}$,
\[
    (1+\varepsilon) U_t - 7U_t^2/2n \le \E\big[U_{t+1} \mid U_t \big] \le (1+\varepsilon) U_t - U_t^2 / n.
\]
\end{lem}
\noindent
We proceed by deriving an exact expression for the variation. To this end, we introduce the following quantities:
\[
    A\coloneqq H_t\left(1-(1-1/n)^{U_t}\right)+U_t\frac{n-H_t}{n}(1-1/n)^{U_t-1};
\]
\[B\coloneqq H_t(H_t-1)\left(1-2\left(1-1/n\right)^{U_t}+\left(1-2/n\right)^{U_t}\right);\]
\[C\coloneqq U_t(U_t-1)\frac{n-H_t}{n}\frac{n-H_t-1}{n}(1-2/n)^{U_t-2};\]
\[D\coloneqq 2H_tU_t\frac{n-H_t}{n}(1-1/n)^{U_t-1}\left(1-\big(1-1/(n-1)\big)^{U_t-1}\right).\]
\begin{lem}\label{secmomclosed}
Let $t\in \mathbb{N}_0$. Then 
\begin{equation*}
\label{eq:SECONDM}
    \mathbb{E}\left[(U_{t+1}-U_t)^2\mid U_t\right]=A+B+C-D.
\end{equation*}
\end{lem}
\begin{proof}
From~\eqref{eq:diffUt} we obtain that $(U_{t+1}-U_t)^2=X^2_{t+1}+Y^2_{t+1}-2X_{t+1}Y_{t+1}$, and so
\begin{equation}
\label{eq:varianceexpanded}
    \mathbb{E}\left[(U_{t+1}-U_t)^2 \mid U_t\right]
    = \mathbb{E}\left[X^2_{t+1} \mid  U_t\right]
        +\mathbb{E}\left[Y^2_{t+1} \mid U_t\right]
        -2\mathbb{E}\left[X_{t+1}Y_{t+1} \mid U_t\right].
\end{equation}
Recalling~\eqref{sums} and that $X_{t+1,h} = \mathbb{1}[h \in \mathcal{U}_{t+1}]$, we obtain
\[
    X^2_{t+1}
    = \sum_{h\in \mathcal{H}_t}X_{t+1,h}
        + \sum_{\substack{h,h'\in \mathcal{H}_t, h\neq h'}} X_{t+1,h}X_{t+1,h'}.
\]
Note that the distribution of $X_{t+1}$, conditional on $U_t$, is invariant under the choice of $\mathcal{H}_t$ and $\mathcal{U}_t$. We will use this fact without further reference within this proof. Then, by linearity of expectation and~\eqref{eq:EXh} we see that
\begin{equation}\label{eq:secX}
    \mathbb{E}\left[X^2_{t+1} \mid U_t\right]
    = H_t\left(1-(1-1/n)\right)^{U_t} + \sum_{\substack{h,h'\in \mathcal{H}_t, h\neq h'}}\mathbb{P}(h,h'\in \mathcal{U}_{t+1} \mid U_t).
\end{equation}
Similarly, we obtain 
\begin{equation*}
    Y^2_{t+1}=\sum_{u\in \mathcal{U}_t}Y_{t+1,u} + \sum_{\substack{u,u'\in \mathcal{U}_t, u\neq u'}} Y_{t+1,u}Y_{t+1,u'},
\end{equation*}
where $Y_{t+1,u} = \mathbb{1}[u \in \mathcal{H}_{t+1}]$, and so, using~\eqref{eq:EXh2}, we obtain 
\begin{equation}\label{eq:secY}
    \mathbb{E}\left[Y^2_{t+1} \mid U_t \right]
    = U_t\frac{n-H_t}{n}(1-1/n)^{U_t-1} + \sum_{\substack{u,u'\in \mathcal{U}_t, u\neq u'}} \mathbb{P}(u,u'\in \mathcal{H}_{t+1} \mid U_t).
\end{equation}
Note that from~\eqref{eq:secX} and \eqref{eq:secY} we already spot the $A$ term in the statement of the lemma.  Next we compute the probabilities that appear in the sums on the right-hand side of \eqref{eq:secX} and \eqref{eq:secY}. To this end, let $h,h'\in \mathcal{H}_t$ with $h\neq h'$ and note that
\begin{equation*}
    \mathbb{P}(h,h'\in \mathcal{U}_{t+1} \mid U_t)
    = 1-\mathbb{P}(h\notin \mathcal{U}_{t+1} \mid U_t)-\mathbb{P}(h'\notin \mathcal{U}_{t+1} \mid U_t)+\mathbb{P}(h,h'\notin \mathcal{U}_{t+1} \mid U_t).
\end{equation*}
We have already argued in Lemma \ref{firstmomentexact} when deriving (\ref{eq:EXh}) that
\[
    \mathbb{P}(h \notin \mathcal{U}_{t+1} \mid U_t)
    = \mathbb{P}(h'\notin \mathcal{U}_{t+1} \mid U_t)
    = (1-1/n)^{U_t}.
\]
Moreover, $h$ and $h'$ both remain happy if all particles in ${\cal U}_t$ jump to vertices different from $p_{h,t}, p_{h',t}$, which occurs with probability $(1-2/n)^{U_t}$. 
So
\begin{equation}
    \sum_{\substack{h,h'\in \mathcal{H}_t, h\neq h'}}\mathbb{P}(h,h'\in \mathcal{U}_{t+1} \mid U_t)
    = H_t(H_t-1)\left(1-2(1-1/n)^{U_t}+(1-2/n)^{U_t}\right),
\end{equation}
which gives the term $B$ in the statement of the lemma.  Next, let $u,u'\in \mathcal{U}_t$ with $u\neq u'$. Then we claim that
\[
    \mathbb{P}(u,u'\in \mathcal{H}_{t+1} \mid U_t)
    = \frac{n-H_t}{n}\frac{n-H_t-1}{n}(1-2/n)^{U_t-2}.
\]
To see this, note that there are $n-H_t$ vertices that are not occupied by happy particles and to which $u$ could jump to in order to become happy. This leaves $n-H_t-1$ spots where $u'$ could jump to become happy. 
Moreover, in order to actually be happy, none of the remaining $U_t-2$  particles in $U_t$ may jump on the vertices where $u$ and $u'$ moved to. Thus
\begin{equation}
    \sum_{\substack{u,u'\in \mathcal{U}_t, u\neq u'}}\mathbb{P}(u,u'\in \mathcal{H}_{t+1} \mid U_t)
    = U_t(U_t-1)\frac{n-H_t}{n}\frac{n-H_t-1}{n}(1-2/n)^{U_t-2},
\end{equation}
which corresponds to the term $C$ in the statement of the lemma.
It remains to evaluate the term $2\mathbb{E}[X_{t+1}Y_{t+1} \mid U_t]$ in~\eqref{eq:varianceexpanded}, which, by linearity of expectation equals 
\[
    2\sum_{\substack{h \in \mathcal{H}_t,  u \in \mathcal{U}_t}} \mathbb{P}(h\in \mathcal{U}_{t+1}, u\in \mathcal{H}_{t+1} \mid U_t).
\]
Observe that, conditional on $u \in {\cal U}_t$ becoming happy at time $t+1$, there remain $U_t-1$ unhappy particles that could make $h \in {\cal H}_t$ unhappy, and they jump independently to uniformly random positions in $\{1, \dots, n\} \setminus \{p_{u,t+1}\}$. Therefore
\[
    \mathbb{P}(h\in \mathcal{U}_{t+1} \mid U_t, u\in \mathcal{H}_{t+1})
    = 1-(1-1/(n-1))^{U_t-1},
    \quad h \in {\cal H}_t, u \in {\cal U}_t.
\]
Hence, by recalling~\eqref{eq:EXh2}, we obtain
\begin{align*}
    \mathbb{P}(h\in \mathcal{U}_{t+1}, u\in \mathcal{H}_{t+1} \mid U_t)
    & = \mathbb{P}(u\in \mathcal{H}_{t+1} \mid U_t)\left(1-(1-1/(n-1))^{U_t-1}\right) \\
    & = \frac{n-H_t}{n}(1-1/n)^{U_t-1}\left(1-(1-1/(n-1))^{U_t-1}\right).
\end{align*}
Summing over all $h\in{\cal H}_t, u\in{\cal U}_t$ yields the term $D$.
\end{proof}
We proceed with establishing sharp asymptotic bounds for the drift and the variation; later, when we apply Theorem~\ref{thm:diffusionapproximation}, we will exploit these estimates when $U_t = \Theta(n^{1/2})$. In what follows we will make use of the following elementary asymptotic estimate, which can readily be established by considering the Taylor series expansion with a remainder term of the function $x \mapsto (1-x)^N$, $N\in \mathbb{N}$, at $x=0$. 
\begin{fact}
Uniformly for any $x\in [0,1]$ and $N\in \mathbb{N}_0$ we have
\begin{equation}\label{asymp1}
    (1-x)^N=1-xN+\binom{N}{2}x^2-\binom{N}{3}x^3+O(N^4x^4),
\end{equation}
as $N\to \infty$.
\end{fact}
We proceed with deriving an asymptotic expression for the drift. Note that in the next two lemmas we work under the general assumption that the number of particles $M=(1+\varepsilon)n/2$, with $\varepsilon=o(1)$; later we will set $\varepsilon = 2\alpha/n^{1/2} + o(1/n^{1/2})$ so that $M=n/2+\alpha n^{1/2} + o(n^{1/2})$, which is the main focus of the present paper.
\begin{lem}
\label{lem:exponestepchange}
Let $\varepsilon  =\varepsilon(n) = o(1)$, $u:\mathbb{N}\to \mathbb{N}$ and $M = M(n) \coloneqq (1+\varepsilon)n/2 \in \mathbb{N}$. Then, uniformly,
\[
    \E\big[U_{t+1} - U_t \mid U_t=u \big]
    = \varepsilon u-\frac{u^2}{n}\left(\frac74 + \frac{3\varepsilon}4\right) + O\left(\frac{u}{n}+\frac{u^3}{n^2}\right).
\]
\end{lem}
\begin{proof}
Recall that $M-U_t=H_t$. Since $M=(1+\varepsilon){n}/{2}$, it follows from Lemma \ref{firstmomentexact} that 
\begin{multline}
\label{befapproximation}
    \mathbb{E}\left[U_{t+1} \mid U_t = u\right]
    \\=
    (1+\varepsilon)\frac{n}{2}\left(1-\Big(1-\frac1n\Big)^u\right)
    + u\left(\Big(1-\frac1n\Big)^u - \Big(\frac{1-\varepsilon}2+\frac{u}n\Big)\Big(1-\frac1n\Big)^{u-1}\right).
\end{multline}
We use~\eqref{asymp1} to obtain the uniform bounds 
\begin{equation}\label{firstapprox}
    (1-1/n)^u = 1-\frac{u}{n}+\frac{u^2}{2n^2}+O\left(\frac{u}{n^2}+\frac{u^3}{n^3}\right),
    \quad
    (1-1/n)^{u-1}=1-\frac{u-1}{n}+O\left(\frac{u^2}{n^2}\right).
\end{equation}
By approximating the terms in \eqref{befapproximation} with (\ref{firstapprox}), where only a second order estimate is used for the second $(1-n^{-1})^n$ term, and using that $\varepsilon$ is bounded, we obtain, maintaining the order of the terms within \eqref{befapproximation},
\begin{align*}
    \mathbb{E}[U_{t+1} \mid U_t = u]
    & = 
    (1 +\varepsilon)\frac{n}2\left(\frac{u}n-\frac{u^2}{2n^2}+O\Big(\frac{u}{n^2}+\frac{u^3}{n^3}\Big)\right) \\ 
    & \qquad + u\Bigg(1-  \frac{u}n+O\Big(\frac{u^2}{n^2}\Big) - \Big(\frac{1-\varepsilon}2 + \frac{u}n\Big)\Big(1-\frac{u-1}n+O\Big(\frac{u^2}{n^2}\Big)\Big) \Bigg).
\end{align*}
Collecting and cancelling terms then yields the statement.
\end{proof}
The next lemma establishes an asymptotic expression for the variation when~$U_t$ is not too big.
\begin{lem}\label{asymptsecoment}
Let $\varepsilon = \varepsilon(n) = o(1)$ and $u:\mathbb{N}\to \mathbb{N}$ be such that $u =o(n^{2/3})$ and $M=M(n) \coloneqq (1+\varepsilon)n/2 \in \mathbb{N}$. Then, uniformly,
\[
    \E\big[(U_{t+1} - U_t)^2 \mid U_t=u \big] = u+o(\varepsilon u^2+u). 
\]
\end{lem}
\begin{proof}
Recall from Lemma~\ref{secmomclosed} that
\[\mathbb{E}\left[(U_{t+1}-U_t)^2\mid U_t\right]=A+B+C-D,\]
where the terms $A,B,C,D$ are defined just prior to Lemma~\ref{secmomclosed}. We estimate these terms individually, under the assumption that $U_t=u = o(n^{2/3})$. We start with $A$ which, using $H_t=M-U_t=M-u$, can be rewritten as
\begin{equation*}\label{exprA}
    A=(M-u)\left(1-(1-1/n)^{u}\right)+u(n-(M-u))n^{-1}(1-1/n)^{u-1}.
\end{equation*}
Using~\eqref{firstapprox} and recalling that  $M = (1+\varepsilon)n/2$ we obtain 
\begin{equation*}
    1-(1-1/n)^{u} = \frac{u}{n} + O\Big(\frac{u^2}{n^2}\Big),
    ~
    (1-1/n)^{u-1} = 1-\frac{u-1}{n} + O\Big(\frac{u^2}{n^2}\Big)    
\end{equation*}
and
\[u\frac{n-(M-u)}{n}
    = u\frac{1-\varepsilon}{2} + \frac{u^2}n.
\]
Note that, since $u=o(n)$,
\[
    (M-u)\left(1-(1-1/n)^{u}\right)=\frac{1+\varepsilon}{2}u+o(u),
    \quad
    u(n-(M-u))n^{-1}(1-1/n)^{u-1}=\frac{1-\varepsilon}{2}u+o(u);
\]
whence we obtain
\begin{equation}
\label{A}
    A=u+o(u).
\end{equation}
We proceed by evaluating $B$ which, as $U_t=u$, is given by
\[B= (M-u)((M-u)-1)\left(1-2\left(1-1/n\right)^{u}+\left(1-2/n\right)^{u}\right).\]
Let us start by making a useful observation. Note that, by applying~\eqref{asymp1} to the expressions $(1-x)^N$ and $(1-2x)^N$ for $x\in [0,1/2]$ and $N\in \mathbb{N}_0$ we obtain uniformly
\begin{equation}
\label{diffexp}
    1-2(1-x)^N+(1-2x)^N
    = N(N-1)x^2-N(N-1)(N-2)x^3+O(N^4x^4).
\end{equation}
Consequently, using (\ref{diffexp}) we obtain
\begin{equation}\label{expansionbrack}
    1-2\left(1-1/n\right)^{u}+\left(1-2/n\right)^{u}=\frac{u(u-1)}{n^2}-\frac{u^3}{n^3}+O\left(\frac{u^2}{n^3}+\frac{u^4}{n^4}\right).
\end{equation}
Also, using once more the fact that $M = (1+\varepsilon)n/2$ we see that
\[
    (M-u)((M-u)-1)
    =\left(\frac{1+\varepsilon}{2}n-u\right)^2-\frac{1+\varepsilon}{2}n+u
    =\frac{(1+\varepsilon)^2}{4}n^2-un(1+\varepsilon)+ O(u^2 + n).
\]
Then, recalling our assumption $u = o(n^{2/3})$, after multiplying the expression that we have just derived for $(M-u)((M-u)-1)$ with the expression on the right-hand side of~\eqref{expansionbrack} and truncating the computations at terms that are $o(u)$ we obtain
\begin{equation}\label{B}
    B=\frac{(1+\varepsilon)^2}{4}u(u-1)-\frac{u^3}{n}\left(\frac{5}{4}+\frac{3\varepsilon}{2}+\frac{\varepsilon^2}{4}\right)+o(u).
\end{equation}
Next we consider the terms $C$ and $D$. Recall that, since $U_t=u$, 
\[C= u(u-1)\frac{n-(M-u)}{n}\frac{n-(M-u)-1}{n}(1-2/n)^{u-2}.\]
Since $M=n(1+\varepsilon)/2$ we can write
\[\frac{n-(M-u)}{n}\frac{n-(M-u)-1}{n}=\frac{(1-\varepsilon)^2}{4}+\frac{u}{n}(1-\varepsilon)+O\left(\frac{u^2}{n^2}+ \frac{1}{n}\right).\]
Then, since $u=o(n^{2/3})$, truncating at $o(u)$ terms we obtain
\begin{equation*}
    u(u-1)\frac{n-(M-u)}{n}\frac{n-(M-u)-1}{n}=u(u-1)\frac{(1-\varepsilon)^2}{4}+\frac{u^3}{n}(1-\varepsilon)+o(u).
\end{equation*}
Moreover, using \eqref{asymp1} we see that
\begin{equation*}
    (1-2/n)^{u-2}=1-\frac{2u}{n}+O\left(\frac{u^2}{n^2}+ \frac{1}{n}\right).
\end{equation*}
By multiplying the last two expressions and using $u = o(n^{2/3})$ we arrive at
\begin{equation}\label{C}
    C=u(u-1)\frac{(1-\varepsilon)^2}{4}+\frac{u^3}{n}\frac{1-\varepsilon^2}{2}+o(u).
\end{equation}
Finally we evaluate the term $D$ which, as $U_t=u$, can be written as
\[D= \left(1-(1-1/(n-1))^{u-1}\right)2(M-u)u\frac{n-(M-u)}{n}(1-1/n)^{u-1}= D_1D_2(1-1/n)^{u-1},\]
where we set
\[D_1\coloneqq 1-(1-1/(n-1))^{u-1}\text{ and } D_2\coloneqq 2(M-u)u\frac{n-(M-u)}{n}.\]
Since $M=(1+\varepsilon)n/2$, we can write
\begin{equation}\label{factD}
D_2=\left(n(1+\varepsilon)-2u\right)u\left(\frac{1-\varepsilon}{2}+\frac{u}{n}\right)=\frac{u n}{2}(1-\varepsilon^2)+2\varepsilon u^2-\frac{2u^3}n.
\end{equation}
Moreover, again using~\eqref{asymp1}, we claim that 
\[D_1=\frac{u-1}{n} - \frac{u^2}{2n^2}+O\left(\frac{u^3}{n^3}+\frac{u}{n^2}\right).\]
To see this, let us start by noticing that, thanks to~\eqref{asymp1}, we can write
\[
    D_1 = \frac{u-1}{n-1}-\frac{u^2}{2(n-1)^2}+O\left(\frac{u^3}{n^3}+\frac{u}{n^2}\right).
\]
Observing that $n/(n-1)=1+1/n+O(1/n^2)$ then establishes the claim about $D_1$.
Analogous computations show that the remaining $(1-1/n)^{u-1}$ term that appears in the expression for $D$ satisfies 
\[
    (1-1/n)^{u-1}
    = 1-\frac{u-1}n + O\left(\frac{u^2}{n^2}\right)
\]
and hence after a little algebra we arrive at 
\begin{equation*}
    D_1(1-1/n)^{u-1}=\frac{u-1}{n}-\frac{3u^2}{2n^2}+O\left(\frac{u}{n^2}+\frac{u^3}{n^3}\right).
\end{equation*}
Multiplying the last expression with \eqref{factD} and using that $u = o(n^{2/3})$ we obtain
\begin{equation*}\label{D}
    D=u(u-1)\frac{1-\varepsilon^2}{2}-\frac{u^3}{n}\left(\frac{3}{4}(1-\varepsilon^2)-2\varepsilon\right)+o(u).
\end{equation*}
Consequently, by combining this bound with~\eqref{C} we obtain
\begin{align*}\label{C-D}
    C-D&=u(u-1)\left(\frac{(1-\varepsilon)^2}{4}-\frac{1-\varepsilon^2}{2}\right)+\frac{u^3}{n}\left(\frac{1-\varepsilon^2}{2}+\frac{3}{4}(1-\varepsilon^2)-2\varepsilon\right)+o(u)\\
    &=-u(u-1)\left(\frac{1}{4}+\frac{\varepsilon}{2}-\frac{3\varepsilon^2}{4}\right)+\frac{u^3}{n}\left(\frac{5}{4}(1-\varepsilon^2)-2\varepsilon\right)+o(u).
\end{align*}
Summing $A$ and $B$ as given in (\ref{A}) and (\ref{B}), respectively, we obtain
\begin{align*}
    A+B+C-D&=u+\frac{(1+\varepsilon)^2}{4}u(u-1)-\frac{u^3}{n}\left(\frac{5}{4}
    +\frac{3\varepsilon}{2}+\frac{\varepsilon^2}{4}\right)\\
    &-u(u-1)\left(\frac{1}{4}+\frac{\varepsilon}{2}-\frac{3\varepsilon^2}{4}\right)+\frac{u^3}{n}\left(\frac{5}{4}(1-\varepsilon^2)-2\varepsilon\right)+o(u)\\
    &=u(1-\varepsilon^2)+u^2\varepsilon^2-\frac{u^3}{n}\left(\frac{5}{4}+\frac{3\varepsilon}{2}+\frac{\varepsilon^2}{4}-\frac{5}{4}(1-\varepsilon^2)+2\varepsilon\right)+o(u)\\
    &=u-u\varepsilon^2+u^2\varepsilon^2-\frac{u^3}{n}\left(\frac{7\varepsilon}{2}+\frac{3}{2}\varepsilon^2\right)+o(u).
\end{align*}
But 
$u^3\varepsilon/n,u^2\varepsilon^2=o(\varepsilon u^2)$
and so the proof is finished. 
\end{proof}


\section{The Early Steps}
\label{sec:earlySteps}

In this section we provide, roughly speaking, tight bounds for the typical trajectory of the process $(U_t)_{t \in  \mathbb{N}_0}$ of unhappy particles until the first step where $U_t$ drops to a value below $n^{1/2}/\delta$, where $\delta > 0$ is (arbitrarily) small. The main result of this section is the following statement. 
\begin{lem}
\label{lem:earlyStepsSummary}
There is a $\chi\in\mathbb{R}$ such that the following is true.
Let $\epsilon > 0, \alpha \in \mathbb{R}$, $M = n/2 + \alpha n^{1/2} + o(n^{1/2}) \in \mathbb{N}$, and let for $\delta > 0$
\[
    T_{n, M,\delta} \coloneqq \inf \left\{t > 0: U_t \le n^{1/2} / \delta\right\}.
\]
Then, for all sufficiently small $\delta > 0$ and all sufficiently large $n$, with probability at least $1-\delta$,
\[U_{T_{n,M,\delta}}\sim n^{1/2}/\delta, \quad
    T_{n,M,\delta} = (1\pm \epsilon) \frac{4}{7} \delta n^{1/2},
    \quad\text{and}\quad 
    \sum_{0 \le t \le T_{n,M,\delta}} U_t = \frac{2}{7}n \ln n + \frac{4}{7}n \ln{\delta} + \chi n \pm \epsilon n.    
\]
\end{lem}
\noindent
In particular, (roughly) $\frac47\delta n^{1/2}$ steps are required to drop below $n^{1/2}/\delta$ unhappy particles, and at this step the accumulated number of unhappy particles, which corresponds to the total number of jumps, is (roughly) $\frac27n\ln n$.

The rest of this section is devoted to the proof of Lemma~\ref{lem:earlyStepsSummary}. As a preparation we first establish that the (conditional) distribution of the number $U_{t+1}$ of unhappy particles, given $U_t$, is subgaussian. This simple fact will be useful in several occasions, and it will allow us to apply the machinery from Section~\ref{ssec:preliminariesmartingales}.
\begin{lem}
\label{lem:subgaussian}
Let $M = M(n) \in \mathbb{N}$. Then, conditional on $U_t$, the random variable $U_{t+1}$ is subgaussian with parameter $64U_{t}$. In particular,
\[
    \mathbb{P}\big[U_{t+1} - \mathbb{E}[U_{t+1}\mid U_t] \ge  \lambda \mid U_t\big],
    ~~\mathbb{P}\big[U_{t+1} - \mathbb{E}[U_{t+1}\mid U_t] \le  -\lambda \mid U_t\big]    
    \le
    e^{-\lambda^2/8U_t},
    \quad
    \lambda > 0.
\]
\end{lem}
\begin{proof}
Let us assume without loss of generality that the set of unhappy particles at time $t$ is $\mathcal{U}_t=\{p_1,\ldots,p_{U_t}\}$. 
Define the Doob martingale 
\[
    M_i\coloneqq \mathbb{E}_{p_{1,t+1},\dots,p_{i,t+1}}[U_{t+1} \mid U_{t}], \quad 1\le i \le U_t,
\]
and set $M_0=\mathbb{E}[U_{t+1} \mid  U_t]$. In words, $M_i$ is the (conditional) expectation of $U_{ t+1}$ given $U_{t}$, knowing the values of $p_{1,t+1},\dots,p_{i,t+1}$, i.e.\ the movements of the first $i$ particles that were unhappy at step~$t$. 
Clearly, $M_{U_t}=U_{t+1}$ and $|M_{i+1}-M_{i}|\leq 2$ for $1\leq i\leq U_{ t}$ since, by changing the position to which any one of the $U_t$ particles moves to, we alter $U_{t+1}$ by at most $2$. 
Therefore, we can apply the Azuma-Hoeffding inequality from Theorem \ref{azumahoeffding} and obtain
\[
    \mathbb{P}\big[U_{t+1} - \mathbb{E}[U_{t+1}\mid U_t] \ge \lambda \mid U_t\big]
    \le \exp\left(-\lambda^2/8U_t\right);
\]
the other bound is obtained analogously. The claim then follows from Theorem~\ref{subG}.
\end{proof}
An important and useful consequence of the previous statement is the powerful property that the number of unhappy particles $U_{t+1}$ is typically always `close' to its conditional expected value $\mathbb{E}[U_{t+1} \mid U_t]$, as stated in the following lemma.
\begin{lem}
\label{lem:roughcontrolfulltrajectory}
Let $\alpha \in \mathbb{R}$ and $M = n/2 + \alpha  n^{1/2} + o(n^{1/2}) \in \mathbb{N}$. Then $\Prob({\cal E}) = 1 - o(1/n)$, where
\[
    {\cal E}
    \coloneqq
    \bigcap_{t \ge 0} {\cal E}_t
    \quad \text{and} \quad
    \mathcal{E}_t\coloneqq \left\{
        \big|U_{t+1} - \E[U_{t+1} \mid U_t]\big| \le 4\sqrt{U_t \ln n}
    \right\}.
\]
\end{lem}
\begin{proof}
Theorem~\ref{cor:tailbounds} asserts that $T_{n,M} \le t_n \coloneqq n^{1/2} \ln^2 n$ with probability $1 - o(1/n)$, so it suffices to consider $t \le t_n$ only in the following sense. Since $\mathcal{E}^c_t=\emptyset$ for $t\geq T_{n,M}$, we obtain 
\begin{equation*}
    \mathbb{P}(\mathcal{E}^c)
    = \mathbb{P}\Big(\bigcup_{t\leq T_{n,M}}\mathcal{E}^c_t\Big)\leq \mathbb{P}\Big(\bigcup_{t\leq t_n}\mathcal{E}^c_t\Big) + \mathbb{P}(T_{n,M}>t_n)
    =\mathbb{P}\Big(\bigcup_{t\leq t_n}\mathcal{E}^c_t\Big)+o(1/n),
\end{equation*}
and hence we can focus on the events $\mathcal{E}^c_t$ for $t\leq t_n$ only. It follows from Lemma~\ref{lem:subgaussian} that 
\[
    \mathbb{P}\left(|U_{t+1} - \E[U_{t+1} \mid U_t]\big| \ge 4\sqrt{U_t \ln n} \mid U_t\right)
    \le
    2\exp\left(-\frac{16 U_t \ln n}{8U_t}\right) = 2n^{-2}.
\]
A union bound then yields $\mathbb{P}({\cal E}) = 1 - o(1/n) - 2t_nn^{-2} = 1-o(1/n)$, as claimed.
\end{proof}
We proceed with the proof of Lemma~\ref{lem:earlyStepsSummary}. For any $0 < \delta,\eta < 1$ set 
\begin{equation}
\label{eq:istar}
    i^* = i^*(\delta, \eta)
    \coloneqq
    \left\lfloor
    \left(\frac{\ln(\delta n^{3/10})}{\eta}\right)^3
    \right\rfloor
    \text{ and }
    u_i \coloneqq u_i(\delta,\eta)=\delta^{-1} n^{1/2}\exp\left(\eta (i^*-i)^{1/3}\right),~0\le i \le i^*.
\end{equation}
Note that $u_{i^*} = n^{1/2}/\delta$ and $e^{-\eta}n^{4/5}\le u_0 \le n^{4/5}$. 
Moreover, let $T_0 \coloneqq
    \inf\{t \ge 0: U_{t} < u_0\}$,
    and define recursively the stopping times
\[
    T_i \coloneqq
    \inf\big\{t \ge 0: U_{T_0 + T_1 + \dots + T_{i-1} + t} < u_i\big\},
    \quad
    1 \le i \le i^*.
\]
These times play the role of `checkpoints' for the evolution of the number of unhappy particles. Indeed, at time $T_0$ we have, for the first time, less than $u_0$ unhappy particles.
Then, after~$T_1$ additional steps, the number of unhappy particles falls for the first time below $u_1$, \dots, and after~$T_i$ additional steps, the number of unhappy particles falls for the first time below $u_i$. Eventually, after $T_0 + \dots + T_{i^*}$ steps the number of unhappy particles will fall for the first time below $n^{1/2}/\delta$, so that $T_{n,M,\delta} = T_0 + \dots + T_{i^*}$. The definition of the $u_i$'s may look a bit cumbersome, but, as we will see, it is chosen with great care so that it allows us to control the total deviation of all the~$U_t$'s  simultaneously from their typical trajectory.

The rest of the section is structured as follows. First, we study $T_0$ by showing that $U_t$, as long as it is rather largish, stays `very close' to a deterministic value that can be obtained by iterating a (deterministic) function $t$ times. This statement is in a specific sense best possible, as after some point in time the number of unhappy particles starts fluctuating too much. Moreover, we obtain fine control of the sum of the iterates of this function. Then, in Subsection~\ref{ssec:Ti} we use martingale arguments to show a general statement that allows us to bound the number of steps that are needed to drop from a given number $u$ of unhappy particles to a slightly smaller number $(1-\eta)u$. Then we apply this statement with (essentially) the ratios $1-\eta \approx u_{i+1}/u_i$ to gain sufficient control of all $T_i$'s. Finally, in Subsection~\ref{ssec:granfinal} we combine our findings to prove Lemma~\ref{lem:earlyStepsSummary}.

\subsection{Analysis of $T_0$} \label{ssec:veryEarlySteps}

In order to study the beginning of the process we show that the number of unhappy particles at step $t$, as long as $t$ is not too large, stays close to it's iterated expectation, or, more precisely, to an asymptotic version thereof, using Lemma~\ref{eq:FIRSTM}.
\begin{lem} 
\label{lem:veryEarly}
Let $\alpha \in \mathbb{R}$ and $M = n/2 + \alpha n^{1/2} + o(n^{1/2}) \in \mathbb{N}$. Define the functions
\[
    f:\mathbb{R}\to\mathbb{R},
    \quad
    x \mapsto x + \left(\frac12-x\right)(1 - e^{-x}) - x\left(\frac12 + x\right)e^{-x}.
\]
and $f^{(t)}:\mathbb{R}\to\mathbb{R}$ inductively by $f^{(0)}(x) = x$ and $f^{(t+1)}(x) = f(f^{(t)}(x))$ for any $x \in \mathbb{R}$ and $t \in \mathbb{N}_0$. Then, with probability tending to one,
\[
    U_t = f^{(t)}(1/2) n \pm (t+1)n^{1/2}\ln n, \quad t \ge 0.
\]
\end{lem}
\begin{proof}
We begin with two observations that we shall use later. Since $U_t\le n$, by applying Lemma~\ref{lem:roughcontrolfulltrajectory} we obtain that with probability $1-o(1)$ and with room to spare
\begin{equation}
\label{eq:trajUt}
    U_{t+1} = \mathbb{E}[U_{t+1}\mid U_t] \pm 4n^{1/2}\ln^{1/2} n,
    \quad
    \text{for all}
    \quad
    t\in\mathbb{N}_0.
\end{equation}
Moreover, using Lemma~\ref{eq:FIRSTM} we obtain for all sufficiently large $n$ and for all $0 \le x \le 1$ such that $xn \in \mathbb{N}$ and $t\in\mathbb{N}_0$ \begin{equation*}
    \frac1n\mathbb{E}\big[U_{t+1} - U_t \mid  U_t = xn\big]
    =  (1/2 - x)\left(1-\Big(1-\frac{1}{n}\Big)^{xn}\right) - x(1/2+x)\Big(1-\frac{1}{n}\Big)^{xn-1}
    \pm 3\alpha n^{-1/2}.
\end{equation*}
Note that $(1-n^{-1})^n = e^{-1}(1 \pm n^{-1})$, since $(1-n^{-1})^n$ is increasing and $(1-n^{-1})^{n-1}$ is decreasing (and both approach $e^{-1})$. 
Thus it follows from the definition of $f(x)$ given in the statement of the lemma that, for all sufficiently large $n$,
\begin{equation}
\label{eq:ExpUtsimpler}    
    \frac1n\mathbb{E}\big[U_{t+1} \mid  U_t = xn\big]
    = f(x) \pm (3|\alpha|+1) n^{-1/2},
    \quad
    \text{where}
    \quad
    xn \in \mathbb{N}, t\in\mathbb{N}_0.
\end{equation}
In the rest we will assume that $n$ is so large that we can use~\eqref{eq:ExpUtsimpler} and $\ln n \ge 4|\alpha|$. Moreover, we assume that~\eqref{eq:trajUt} holds. Under these assumptions we will show (by induction on $t$) that 
\[U_t = f^{(t)}(1/2) n \pm (t+1)n^{1/2}\ln n \text{ for all }t \in \mathbb{N}_0,\]
so that the lemma is established.
Since $U_0 = M$ and $f^{(0)}(1/2) = 1/2$ the claim is obviously true for $t = 0$. Next, assume the claim holds for some $t \in \mathbb{N}$. From~\eqref{eq:trajUt} and~\eqref{eq:ExpUtsimpler} together with the induction hypothesis we obtain
\[
    \frac1n U_{t+1} = f\big( f^{(t)}(1/2) \pm (t+1)n^{-1/2}\ln n \big) \pm (3|\alpha|+1) n^{-1/2}.
\]
Note that $f'(x) = e^{-x}(1-5x/2+x^2)$, so that $|f'(x)| \le 1$ for all, say, $x\in[0,1]$. Thus, by Taylor's theorem, $f(x+y) = f(x) \pm y$  for all $x,y$ such that $x,x+y\in[0,1]$. We obtain 
\[
    \frac1n U_{t+1} = f^{(t+1)}(1/2) \pm (t+1)n^{-1/2}\ln n \pm (3|\alpha|+1) n^{-1/2}.
\]
Since $\ln n \ge (3|\alpha|+1)$ the claim is established.
\end{proof}
Towards the proof of Theorem~\ref{numbjumps} the following lemma about the sum of the iterates of $f$ will be very useful.
\begin{lem}\label{lem:iterMean}
Let $f,f^{(t)}$  be as in Lemma~\ref{lem:veryEarly}. Then $f^{(t)}(1/2)$ is decreasing in $t$ and 
there is a $\chi \in \mathbb{R}$ such that, as $N\to \infty$,
\[
f^{(N)}(1/2)=\frac47 N^{-1} + O\big(N^{-2} \ln N\big)
\quad
\mbox{and}
\quad
    \sum_{t \le N} f^{(t)}(1/2) = \frac47\ln N + \chi + O(N^{-1} \ln N).
\]
\end{lem}
\begin{obs}\label{rem:chi}
Numerical investigations show that $\chi = -0.1236(\dots)$. However, the convergence is quite slow, as indicated by our error bounds.
\end{obs}
\begin{proof}
Let us collect some properties of $f$.
Note that $f'(x) = e^{-x}(1-5x/2+x^2)$ and that $1/2$ and $2$ are the roots of $1-5x/2+x^2$. So, $f'$ is positive and decreases on $[0,1/2)$. Since $f(0) = 0$ we obtain $f(x) > 0$ for $x \in (0,1/2]$. Moreover, $f'(0) = 1$ and thus $(f(x)-x)' < 0$ on $(0,1/2
]$; again using that $f(0) = 0$ we obtain $f(x) < x$ in this interval. All in all we have established
\begin{equation}
\label{eq:boundsf}
    0 < f(x) < x, \quad x \in (0,1/2].
\end{equation}
Let us consider the sequence $(f^{(t)}(1/2))_{t \in \mathbb{N}_0}$ which, according to~\eqref{eq:boundsf}, is \emph{decreasing} and \emph{positive}. Since $0$ is the only fixpoint of the (continuous) $f$ in $[0,1/2]$ we obtain
\begin{equation}
\label{eq:limft}
    f^{(t)}(1/2) \to 0 \quad \text{as} \quad t \to \infty.
\end{equation}
A simple computation shows that 
\[
    f(0) = 0, f'(0) = 1, f''(0) = -\frac72
    \quad
    \text{and}
    \quad 
    f'''(x) = e^{-x}(8-13x/2+x^2).
\]
and so there is a $C>0$ such that
\begin{equation}
\label{eq:taylorf}
    f(x) = x - \frac74x^2 \pm Cx^3, \quad x\in[0,1/2].
\end{equation}
Up to now we have collected several properties of $f$ that are actually sufficient to establish the claim. We will proceed with a routine argument, and the interested reader may consult~\cite[Ch.~8.4]{book:deBruijn} for prototypical applications. Let us abbreviate $x_t = f^{(t)}(1/2)$ and define $y_t = x_t^{-1}$. We obtain from~\eqref{eq:taylorf} and~\eqref{eq:limft}, as $t \to \infty$,
\[
    y_{t+1}
    = y_t \Big(1 - \frac74 x_t + O(x_t^2)\Big)^{-1}
    = y_t \Big(1 + \frac74 x_t + O(x_t^2)\Big)
    = y_t + \frac74 + O(x_t).
\]
Since $x_t \to 0$ as $t\to\infty$ we obtain that, say, $y_{t+1} - y_t \ge 1$ for all sufficiently large $t$. Thus $y_t \ge t/2$ and so $x_t \le 2/t$ for all sufficiently large $t$. This, in turn, implies $y_t = \frac{7t}{4} + O(\ln t)$, and thus
\[
    x_t = f^{(t)}(1/2) = \frac47 t^{-1} + O\big(t^{-2} \ln t\big),
    \quad
    t\to\infty.
\]
In other words, if we set $g(t) \coloneqq f^{(t)}(1/2) - 4/7t$, then $|g(t)| = O(t^{-2} \ln t)$. Thus $\sum_{t \ge 0} g(t)$ is absolutely convergent, and since $\sum_{1 \le t \le N} t^{-1} = \ln N + \gamma + O(N^{-1})$ for a $\gamma \in \mathbb{R}$ the claim is established.
\end{proof}
By combining the previous facts we readily obtain the likely value of $T_0$ and $U_{T_0}$. 
\begin{cor}
\label{cor:T0}
Let $\alpha \in \mathbb{R}, M = n/2 + \alpha n^{1/2} + o(n^{1/2}) \in \mathbb{N}$ and $0<\delta, \eta<1$. Then, with probability tending to one as $n\to\infty$,
\[
    T_0 = (1\pm 2\eta) \frac{4}{7} n^{1/5}
    \quad
    \text{and}
    \quad
    U_{T_0} = (1+o(\ln^{-4}n)) u_0.
\]
\end{cor}
\begin{proof}
Recall from (\ref{eq:istar}) that $e^{-\eta}n^{4/5}\leq u_0\leq n^{4/5}$.
By Lemma~\ref{lem:veryEarly} with probability tending to one as $n\to \infty$
\[
U_t = f^{(t)}(1/2) n \pm (t+1)n^{1/2}\ln n, \quad t \ge 0.
\]
We will assume that this event holds for the remainder of the proof. Setting $N_- \coloneqq 4(1-\eta)n^{1/5}/7$ by Lemma~\ref{lem:iterMean}, we have for every $t\le N_-$ that 
\[
    U_{t}\ge (1+\eta)n^{4/5}+O\left(n^{3/5}\ln{n}\right)\pm 2n^{7/10} \ln n
    > n^{4/5}
    \ge u_0
\]
for all large enough $n$. Moreover, setting $N_+ \coloneqq 4(1+2\eta)n^{1/5}/7$,
\[
    U_{N_+}\le \frac{n^{4/5}}{1+2\eta}+O\left(n^{3/5}\ln{n}\right)\pm 2n^{7/10} \ln n
    < e^{-\eta} n^{4/5}
    \le u_0
\]
for all large enough $n$. Therefore $4(1-\eta)n^{1/5}/7\leq T_0\leq 4(1+2\eta)n^{1/5}/7,$
establishing the first claim. Concerning the second statement, it follows again from Lemma~\ref{lem:iterMean} that, whenever $t \in [N_-, N_+]$,
\[
    |U_{t+1} - U_t|
    \le O(nt^{-2}\ln t) + 2(t+2)n^{1/2}\ln n
    = o(n / t \ln^4{n})
\]
and so $U_{t+1} = (1+o(\ln^{-4}{n})) U_t$ for all such $t$.
\end{proof}

\subsection{Analysis of $T_i, 1\le i \le i^*$}
\label{ssec:Ti}

Our next lemma  states that if we begin with a sufficiently large number $u$ of unhappy particles, then typically the number of steps that are required to drop  below $(1-\eta)u$ unhappy particles is roughly $4\eta n / 7u$ and, during those steps, the number of unhappy particles never exceeds $(1+\eta)u$. In the sequel, this will allow us to study $T_i$, by selecting $u_i = (1-\eta_i)u_{i-1}$, see~\eqref{eq:istar}. In the proof of Lemma~\ref{lem:martingalepair} below we construct two processes, a supermartingale and a submartingale, which bound from below and above, respectively, the number of unhappy particles; then we analyse these processes by exploiting that $U_{t}$, conditional on $U_{t-1}$, is subgaussian, see Lemma~\ref{lem:subgaussian}. 
\begin{lem}
\label{lem:martingalepair}
Let $0 < \eta < 1/16$. Then there is an $n_0 \in \mathbb{N}$ such that for all $n \ge n_0$ the following is true.
Let $\alpha \in \mathbb{R}$ and $M = n/2 + \alpha n^{1/2} + o(n^{1/2}) \in \mathbb{N}$. Moreover, let $u:\mathbb{N}\to\mathbb{N}$ be such that $3|\alpha|\eta^{-1}n^{1/2} \le u \le n/\ln^{1/2} n$ and, for $t_0 \in \mathbb{N}$, let
$$
    T
    = T(u,t_0)
    \coloneqq \inf\big\{t\in \mathbb{N}_0: U_{t_0+t}\not\in [(1-\eta)u,(1+\eta)u]\big\}.
$$
Then, conditional on $U_{t_0}=u$, with probability at least $1 - 8\exp\left(- \eta^3 u^2/800 n\right)$, 
\[
    T = (1  \pm 5\eta) \frac{4\eta n}{7u}
    \quad\text{and}\quad
    U_{t_0 + T} = (1 - \eta)(1 \pm 12\eta^2)u.
\]
\end{lem}
\begin{proof}
Since $(U_t)_{t\in\mathbb{N}_0}$ is a Markov chain, we assume without loss of generality within the proof that $t_0 = 0$; in particular, with slight abuse of notation, $U_0 = u$. Set $\varepsilon \coloneqq {2M}/{n}-1 \sim {2\alpha}n^{-1/2}$ and define, for $t\in \mathbb{N}_0$, the two discrete-time random processes
\[
    W_t^{+}\coloneqq U_{t\wedge T}+\frac74\frac{(1+\eta)^3 u^2}{n}(t \wedge T)
    \quad \mbox{and}\quad
    W_t^{-}
    \coloneqq U_{t\wedge T}+\frac74\frac{(1-\eta)^3 u^2}{n}(t \wedge T). 
\]
For $t\in\mathbb{N}_0$, let~$\mathcal{F}_t$ be the $\sigma$-algebra generated by the process of unhappy particles until step $t$, and set ${\cal F} \coloneqq ({\cal F}_t)_{t\in \mathbb{N}_0}$. We will argue that $(W_t^+)_{t\in\mathbb{N}_0}$ is a submartingale and that $(W_t^-)_{t\in\mathbb{N}_0}$ is a supermartingale (both with respect to $\cal F$).
By applying Lemma~\ref{lem:exponestepchange} we see that there is a
$n_0 \in \mathbb{N}$ such that for all $n\ge n_0$ we have for every $3|\alpha| \eta^{-1}n^{1/2} \le u \le n/\ln^{1/2} n$ that
\[
    \E\big[U_{t+1} - U_t \mid U_t=u \big]
    = \varepsilon u-\frac{u^2}{n}\left(\frac74 + \frac{3\varepsilon}4\right) + O\left(\frac{u}{n}+\frac{u^3}{n^2}\right)
    = - \big(1 \pm \eta)\frac{7}{4}\frac{u^2}{n}.
\]
In what follows we will always assume that $n \ge n_0$.
Note that, by the Markov property of $(U_t)_{t \ge 0}$, $\mathbb{E}[W_t^-\mid \mathcal{F}_{t-1}] = \mathbb{E}[W_t^-\mid U_{t-1}]$. Since $(1-\eta)u\leq U_{t-1}\leq (1+\eta)u$ if $t\le T$, we obtain under this assumption that 
\begin{align}
\label{supmgcalculation}
    \nonumber\mathbb{E}[W_t^-\mid \mathcal{F}_{t-1}]
    &=\mathbb{E}[U_t\mid U_{t-1}] + \frac74\frac{(1-\eta)^3 u^2}{n}t\\
    \nonumber&\le U_{t-1}-(1- \eta)\frac{7}{4}\frac{U_{t-1}^2}{n}+\frac74\frac{(1-\eta)^3 u^2}{n}t\\
    &\leq U_{t-1}-(1- \eta )\frac{7}{4}\frac{(1-\eta)^2 u^2}{n}+\frac74\frac{(1-\eta)^3 u^2}{n}t=W_{t-1}^-.
\end{align}
On the other hand, $W_t^-=W_{t-1}^-$ when $t>T$, showing that $W_t^-$ is a supermartingale. An analogous argument, using again that $(1-\eta)u\leq U_{t-1} \le (1+\eta)u$ when $t \le T$, shows that $\mathbb{E}[W_t^+\mid \mathcal{F}_{t-1}]\ge W_{t-1}^+$ and therefore $W_t^+$ is a submartingale.  Having established the super/sub-martingale nature of $W^-$ and $W^+$ we claim that
\begin{equation}\label{LBTu}
    \mathbb{P}(T < t^-)
    \le 2\exp\left(-\frac{\eta^3 u^2}{800 n}\right),
    \quad
    \text{where}
    \quad
    t^-
    =t^-(u,n,\eta)
    \coloneqq \frac{4\eta n}{7u} \frac{1-\eta}{(1+\eta)^3}.
\end{equation}
To see this, note that if $T<t^-$ then either $U_{t^-\wedge T}>(1+\eta)u$ or $U_{t^-\wedge T}<(1-\eta)u$ and hence
\begin{equation}\label{LBTusplit}
    \mathbb{P}(T<t^-)\leq \mathbb{P}(U_{t^-\wedge T}>(1+\eta)u)+\mathbb{P}(U_{t^-\wedge T}<(1-\eta)u).
\end{equation}
We start by estimating the probability that $U_{t^-\wedge T} > (1+\eta)u$. To this end we will utilize the supermartingale $W^-_{t^-}$ since, roughly speaking, it is `hard' for a supermartingale to become \textit{large} 
and hence the probability of interest should be (very) small.
Similarly, we will estimate the probability that $U_{t^-\wedge T}$ is smaller than $(1-\eta)u$ by by utilizing $W^+_{t^-}$; this will give us a good bound on the probability of interest since it is `hard' for a submartingale to become \textit{small}.
Following this discussion and using that $W^-_0=U_0=u$ we obtain
\begin{equation}\label{supmgLB}
    \mathbb{P}(U_{t^-\wedge T}>(1+\eta)u)
    = \mathbb{P}\Big(W^-_{t^-}-\frac74\frac{(1-\eta)^3u^2}{n}(t^-\wedge T)> (1+\eta)u\Big)
    \le \mathbb{P}\big(W^-_{t^-}-W^-_0> \eta u\big).
\end{equation}
Lemma~\ref{lem:subgaussian} asserts that $U_t$, conditional on $U_{t-1}$, is a subgaussian random variable with parameter $64U_{t-1}$. This in turn implies that, for every $t\in \mathbb{N}$, the random variable $W_t^{-}$ conditional on $U_{t-1}$ is subgaussian with parameter $64(1+\eta)u \le 100 u$ (with room to spare); the same is true for $W_t^+$ conditional on $U_{t-1}$.
Thus, using Lemma~\ref{lem:martingale} we obtain that
\[
    \mathbb{P}\big(W^-_{t^-}-W^-_0\geq \eta u\big)\leq \exp\left(-\frac{\eta^2 u^2}{ 200 u t^-}\right),
\]
and substituting the value of $t^-$ we see that the last expression is at most half of the upper bound on $\mathbb{P}(T < t^-)$ in \eqref{LBTu}. It remains to bound the second probability on the right-hand side of \eqref{LBTusplit}. Note that
\begin{align*}
    \mathbb{P}(U_{t^-\wedge T}<(1-\eta)u)
    &= \mathbb{P}\Big(W^+_{t^-}-W^+_0<- \eta u+\frac74\frac{(1+\eta)^3u^2}{n}(t^-\wedge T)\Big)\\
    &\le \mathbb{P}\Big(W^+_{t^-}-W^+_0<- \eta u + \frac74\frac{(1+\eta)^3u^2}{n}t^-\Big)\\
    &=\mathbb{P}\big(W^+_{t^-}-W^+_0< -\eta u+\eta u-\eta^2u\big)\\
    &\leq  \mathbb{P}\big(-(W^+_{t^-}-W^+_0)> \eta^2 u\big).
\end{align*}
Using once more Lemmas~\ref{lem:martingale} and~\ref{lem:subgaussian} we see that the last expression is at most
$\exp\big(-\frac{7\eta^3 u^2}{800 n}\big)$, 
which is also at most half of the expression in \eqref{LBTu}, completing its proof. 
Our next goal is to show
\begin{equation}\label{LBTl}
    \mathbb{P}(T>t^+)\leq \exp\left(-\frac{\eta^3 u^2 }{800 n}\right),
    \quad
    \text{where}
    \quad
    t^+=t^+(u,n,\eta)\coloneqq \frac{4 \eta n}{7u} \frac{1+\eta}{(1-\eta)^3}>t^-.
\end{equation}
To see this, observe first that if $T>t^+$, then $(1-\eta)u < U_{t^+} < (1+\eta)u$. So, obviously,
\[
    \mathbb{P}(T>t^+)
    \leq \mathbb{P}\big(U_{t^+} > (1-\eta)u, T>t^+\big).
\]
Note that it doesn't look necessary to maintain the event $\{T>t^+\}$, but we will actually need it in a moment. 
Expressing $U_t$ in terms of $W^-_t$ once more we obtain 
\[
    \mathbb{P}\big(U_{t^+} > (1-\eta)u, T>t^+\big)
    \leq \mathbb{P}\Big(W^-_{t^+}-W^-_0\geq -\eta u +\frac74\frac{(1-\eta)^3 u^2}{n}t^+\Big)
    =\mathbb{P}\Big(W^-_{t^+}-W^-_0\geq \eta^2 u \Big).
\]
Hence, using again Lemmas~\ref{lem:martingale} and~\ref{lem:subgaussian}, we obtain 
\[
    \mathbb{P}(T > t^+)
    \le \exp\left(-\frac{7u^2 \eta^3(1-\eta)^3}{800 n (1+\eta)}\right)
    \le \exp\left(-\frac{u^2 \eta^3}{800 n}\right),
\]
So far we have established that
\begin{equation}\label{firstpartlem61}
    \mathbb{P}(t^-\leq T\leq t^+)
    \ge 1- 3\exp\left(-\frac{\eta^3 u^2}{800 n}\right).
\end{equation}
To complete the proof of the lemma, it remains to show that $U_T = (1-\eta)(1 \pm 12 \eta^2)u$  with the required probability bound indicated in the statement of the lemma. Note that either $U_T\ge(1+\eta)u$ or $U_T\le(1-\eta)u$. Thus, proceeding as in \eqref{supmgLB}, we obtain that
\begin{align*}
    \nonumber\mathbb{P}(U_T\ge(1+\eta)u)
    &\le \mathbb{P}\big(W^-_T - W^-_0 \ge \eta u\big)
    \le \mathbb{P}\big(W^-_{T\wedge t^+}-W^-_0 \ge \eta u\big) + \mathbb{P}(T > t^+).
\end{align*}
Using Lemmas~\ref{lem:martingale} and~\ref{lem:subgaussian} we obtain 
\[
    \mathbb{P}\big(W^-_{T\wedge t^+}-W^-_0 \ge \eta u\big)
    \le \exp\left(-\frac{7\eta u^2}{16\cdot 200 n}\right)
    \le \exp\left(-\frac{\eta^3 u^2 }{800 n}\right),
\]
and so, together with~\eqref{LBTl}, we arrive at 
\begin{equation}\label{eq:lowerexit}
    \mathbb{P}(U_T < (1+\eta)u)
    \ge 1-2\exp\left(-\frac{\eta^3 u^2}{800 n}\right).
\end{equation}
The last missing part is to show that $\mathbb{P}\big(U_T \ge (1-\eta)(1-12\eta^2)u\big)$ is large. To this end we utilize for a last time the submartingale $W^+_t$. Therefore,
\begin{align}\label{eqqqq}
    \nonumber\mathbb{P}\big(&U_T < (1-\eta)(1-12\eta^2)u\big)\\
    &\le \mathbb{P}\left(W^+_{T\wedge t^+}-W^+_0 < -\eta u - 12\eta^2 u +12\eta^3 u +\frac74\frac{(1+\eta)^3 u^2}{n} t^+\right) + \mathbb{P}(T>t^+).
\end{align}
Our assumption $0<\eta<1/16$ guarantees that $(1+\eta)^4\le 1+5\eta$ and $(1-\eta)^{-3}\le (1-3\eta)^{-1}\le 1+4\eta$. Thus
\[
    -\eta - 12\eta^2 +12\eta^3 + \frac74\frac{(1+\eta)^3 u}{n}t^+
    \le -\eta - 11\eta^2 + \eta(1+10\eta)
    \le -\eta^2,
\]
leading to the bound
\[
    \mathbb{P}\left(W^+_{T\wedge t^+} - W^+_0 < -\eta u - 12\eta^2 u+ 12\eta^3 u +\frac74\frac{(1+\eta)^3 u^2}{n}t^+\right)
    \le \mathbb{P}\big(W^+_{T\wedge t^+} - W^+_0 < -\eta^2 u\big).
\]
Using once more Lemmas~\ref{lem:martingale} and~\ref{lem:subgaussian} we see that the last probability is at most $\exp(-\eta^3 u^2 / 800 n)$, and together with \eqref{LBTl} and \eqref{eqqqq}
we conclude that 
\[
    \mathbb{P}\big(U_T<(1-\eta)(1-12\eta^2)u \big)
    \le 2\exp\left(-\frac{\eta^3 u^2}{800 n}\right).
\]
Combining this with \eqref{firstpartlem61} and \eqref{eq:lowerexit} gives the result.
\end{proof}
In the next step we apply Lemma~\ref{lem:martingalepair} iteratively for $1\le i \le i^*$ to obtain bounds for the $T_i$'s and the corresponding number of jumps.
\begin{lem}
\label{lem:si}
Let $\epsilon > 0, \alpha \in \mathbb{R}$ and $M = n/2 + \alpha n^{1/2} + o(n^{1/2}) \in \mathbb{N}$. For all sufficiently small $\delta > 0$ and all sufficiently large $n$ the following holds with probability at least $1-\delta$. Set $\eta \coloneqq 12 \delta^{1/9}$ and $i^* = i^*(\delta, \eta)$. Then $U_{T_{n,M,\delta}}\sim n^{1/2}/\delta,$
\[  T_{n,M,\delta}
    = \sum_{0\le i\le i^*}T_i
    = T_0 + (1\pm \epsilon) \frac{4}{7} \delta n^{1/2},
    \quad \text{and} \quad
    \sum_{T_0 < t \le T_{n,M,\delta}} U_t
    = \frac6{35}n \ln n+\frac{4}{7}n\ln{\delta} \pm \epsilon n.
\]
\end{lem}
\begin{proof}
Assume, with foresight, that $0< \delta< \min\{|\alpha|^{-3/2},2^{-54}3^{-9}\}$. In addition abbreviate $S_i \coloneqq T_0 + \dots + T_i$ for $0 \le i \le i^*$, so that $S_{i^*} = T_{n,M,\delta}$ and set
$$
    \eta_i
    \coloneqq
    \ln\left(\frac{u_i}{u_{i+1}}\right)
    =
    \eta\big((i^*-i)^{1/3}-(i^*-i-1)^{1/3}\big) \le \eta.
$$
We will  apply Lemma~\ref{lem:martingalepair} iteratively.
For $0 \le i \le i^*-1$ assume that $ \{U_{S_{i}} = \hat{u}_{i}\}$ for some $(1-12\eta_i^2)u_{i}\le \hat{u}_{i}\le u_{i}$.
Set $\hat{\eta}_{i}=1-u_{i+1}/\hat{u}_{i}$ and using $1-x\le e^{-x} \le 1-x+x^2$ when $x \ge 0$ we have
\[
    \hat{\eta}_{i} \le 1-\exp(-\eta_i)\le \eta_i
    \quad
    \text{and}
    \quad
    \hat{\eta}_{i}
    \ge 1 - \frac{\exp(-\eta_i)}{1-12\eta_i^2}
    \ge \frac{\eta_i-13\eta_i^2}{1-12\eta_i^2}
    \ge \eta_i-13\eta_i^2.
\]
Due to our choice of $\delta$ and $\eta$ we have $\hat{\eta}_{i}\le \eta \le 1/16$. Moreover, since $a-b = (a^{1/3} - b^{1/3})(a^{2/3} + (ab)^{1/3} + b^{2/3})$ we obtain $\eta_i\ge \eta /3 (i^*-i)^{2/3}$. Using that $\eta = 12\delta^{1/9}$ we obtain  
\[
    \delta^{-1}
    = \delta^{-1/3}\cdot\delta^{- 2/3}
    \ge 2^6 3^3\eta^{-3} \cdot \max\{|\alpha|, 2^{36}3^{6}\}
    \ge 2^6 3^3 |\alpha| \hat{\eta}_{i}^{-1} \cdot \frac{\hat{\eta}_{i}}{\eta^3}
    \ge 2^6 3^2 |\alpha| \hat{\eta}_{i}^{-1} \cdot \frac{1-13\eta}{\eta^2 (i^*-i)^{2/3}}.
\] 
Using $e^x \ge x^2/2, x\ge 0$, this in turn leads to
\[
    \hat{u}_{i}
    \ge (1-12\eta^2) u_{i}
    = (1-12\eta^2)\delta^{-1}n^{1/2} \exp(\eta (i^*-i)^{1/3})
    \ge \delta^{-1}n^{1/2} \eta^2 (i^*-i)^{2/3}/4
    \ge 3|\alpha|\hat{\eta}_{i}^{-1}n^{1/2}.
\]
We have established that conditional on $\{U_{S_{i}} = \hat{u}_{i}\}$ such that $(1-12\eta_i^2)u_{i}\le \hat{u}_{i}\le u_{i}$ every condition of Lemma~\ref{lem:martingalepair} is met if we choose $\hat{\eta}_{i}=1-u_{i+1}/\hat{u}_{i}$.  
Therefore, using $\eta_i\ge \eta (i^*-i)^{-2/3}/3$ and $e^x\ge x^9/9!$ for $x\ge 0$, with probability at least
$$
    1-8\exp\left(-\frac{\hat{\eta}_{i}^3 \hat{u}_{i}^2} {800 n}\right)
    \ge 1-8\exp\left(-\frac{\eta_i^3 u_{i}^2} {4^4 \cdot 1600 n}\right)
    \ge 1-8\exp\left(-\frac{ \eta^{12} (i^*-i)} {9!\cdot 3^3 \cdot 800 \delta^2}\right)
    \ge 1-8\exp\left(-\frac{i^*-i} {\delta^{2/3}}\right)
$$
the following three events hold: 
\[
    T_{i+1}
    = (1 \pm 5\hat{\eta}_{i})\frac{4\hat{\eta}_{i} n}{7\hat{u}_{i}}
    = (1 \pm 25\eta_i)\frac{4\eta_i n}{7u_{i}},
    \qquad
    u_{i+1} \ge U_{S_{i+1}} \ge (1-12\eta_{i+1}^2)u_{i+1},
\]
and
\begin{equation}
\label{eq:boundsUtT}
    U_t = (1 \pm \hat{\eta}_{i})\hat{u}_{i} = (1 \pm 2\eta_i)u_{i},
    \quad
    S_{i}\le t < S_{i+1}.
\end{equation}
With this at hand we will now consider the intersection of these events for all $0 \le i \le i^*-1$ and appropriately chosen $\hat{u}_i$'s. Indeed, by Corollary~\ref{cor:T0}, for sufficiently large $n$, $u_0 \ge U_{T_0}=(1-o(\ln^{-4}n))u_0\ge (1-12\eta_0^2)u_0$ with probability $1-o(1)$. Thus, by induction on $0 \le i \le i^*-1$ and Lemma~\ref{lem:roughcontrolfulltrajectory},  with probability at least
\[
    1-8\sum_{0 \le i \le i^*-1}\exp\left(-\frac{i^*-i}{\delta^{2/3}}\right)-o(1)
    \ge 1 - 9\exp\left(-\delta^{-2/3}\right),
\]
we have 
\[U_{T_{n,M,\delta}}\ge \mathbb{E}[U_{T_{n,M,\delta}}\mid U_{T_{n,M,\delta}-1}]-4\sqrt{U_{T_{n,M,\delta}-1}},\]
\begin{equation}
\label{eq:intervalsteps}
    T_{i+1} = (1\pm 25\eta_i) \frac{4 \eta_i n}{7 u_{i}},
    ~~\text{and}~~
    U_t = (1 \pm 2\eta_i)u_{i},
    \quad
    \text{for all}
     \quad 
     0\le i \le i^*-1,    S_{i}\le t < S_{i+1}.
\end{equation}
Since $u_{i^*-1}=O(n^{1/2})$ we also have $U_{T_{n,M,\delta}-1}=O(n^{1/2})$, which together with Lemma~\ref{lem:driftcoarsebounds} implies $T_{n,M,\delta}\sim n^{1/2}/\delta$.
In order to prove the remaining statements of the lemma, that is, to estimate the sum of the $T_{i+1}$'s and the corresponding sum of the $U_t$'s, we require suitable bounds for the ratios $\eta_i / u_i$.
For $x\in \mathbb{R}$ we have (using the classical inequality $1+x\leq e^x$)
\begin{equation}\label{eq:expbound}
1-x\le \frac{xe^{-x}}{1-e^{-x}}\le 1.
\end{equation}
Since $\eta_i\le \eta$ and $u_i=e^{\eta_i}u_{i+1}$ for $0\le i \le i^*-1$, \eqref{eq:expbound} implies for $0\le i \le i^*-1$ that
\[
    \frac{\eta_i}{u_i}
    \le \frac{\eta_i}{e^{\eta_i} u_{i+1}} \le  \frac{1}{u_{i+1}}-\frac{1}{u_i}
    \quad\text{and}\quad
    \frac{\eta_i}{u_i}
    \ge (1-\eta_i)\left(1-\frac{1}{e^{\eta_i}}\right)\frac{1}{u_{i+1}}
    \ge (1-\eta)\left(\frac{1}{u_{i+1}}-\frac{1}{u_i}\right).
\]
Using \eqref{eq:intervalsteps} and these bounds, the sum of the $T_{i+1}$'s can be bounded by a telescoping sum. So, for sufficiently large $n$, with probability at least $1 - p_\delta$, where $p_\delta \coloneqq 9\exp(-\delta^{-2/3})$,
\begin{equation}
\label{eq:sumTi}    
    \sum_{0 \le i \le i^*-1}T_{i+1}
    = (1\pm 30\eta) \frac{4n}{7u_{i^*}}
    =(1\pm 30\eta)\frac{4}{7}\delta n^{1/2}.
\end{equation}
Moreover, using \eqref{eq:intervalsteps},
$$
    \sum_{T_0 < t \le S_{i^*}} U_t
    =  \sum_{0 \le i \le i^*-1} (1\pm 2\eta_i) T_i u_i
    = \frac{4}{7} n \sum_{0 \le i \le i^*-1} (1\pm 35\eta_i)\eta_i.
$$
Now, using the definition of the $\eta_i$'s,
\[
    \sum_{0 \le i \le i^*-1}\eta_i
    = \eta(i^*)^{1/3}
    =\frac{3}{10}\ln{n}+\ln{\delta} \pm \eta \quad\mbox{and}\quad \sum_{0 \le i \le i^*-1}\eta_i^2\le \eta^2 \sum_{1 \le i \le i^*}\frac{1}{i^{4/3}}\le 4\eta^2.
\]
Thus, whenever $n$ is sufficiently large
\begin{equation}
\label{eq:sumU}
    \sum_{T_0 < t \le S_{i^*}} U_t
    = \frac6{35}n \ln{n}+\frac{4}{7}n\ln{\delta}\pm 10\eta n
\end{equation}
Note that the error terms $\pm 30\eta$ in ~\eqref{eq:sumTi} and $\pm 10\eta$ in \eqref{eq:sumU} can be made arbitrarily small by choosing $\delta > 0$ appropriately. Moreover, the (error) probability $p_\delta$ can similarly be made arbitrarily small, completing the proof. 
\end{proof}

\subsection{Putting everything together -- Proof of Lemma~\ref{lem:earlyStepsSummary}}
\label{ssec:granfinal}

With all preparations at hand we now prove Lemma \ref{lem:earlyStepsSummary}. Let $\epsilon > 0$ and throughout the proof we will assume that $\delta > 0$ is sufficiently small and $n$ is sufficiently large so that we can apply Lemma~\ref{lem:si}. Set $\eta \coloneqq 12\delta^{1/9}$. Then by Lemmas~\ref{lem:veryEarly} and \ref{lem:si} with probability $1-\delta-o(1)$ the following four events hold:
\[
    U_t = f^{(t)}(1/2) n \pm (t+1)n^{1/2}\ln n, \quad t \ge 0, \quad U_{T_{n,M,\delta}}\sim n^{1/2}/\delta,
\]
\[
    T_{n,M,\delta}
    = T_0 + (1\pm \epsilon) \frac{4}{7} \delta n^{1/2}
    \quad\text{and}\quad
    \sum_{T_0 \le t \le T_{n,M,\delta}} U_t
    = \frac{6}{35}n \ln n+\frac{4}{7}n\ln{\delta} \pm \epsilon n.
\]
For the remainder of the proof assume that the above three events hold. 
Then by Corollary~\ref{cor:T0} we have $T_0=(1\pm 2 \eta)4n^{1/5}/7$ and so
$T_{n,M,\delta}=(1\pm 2\epsilon)4\delta n^{1/2}/7$.
Moreover, Lemma~\ref{lem:iterMean} implies that there is a $\chi \in \mathbb{R}$ such that
\[
    \sum_{0 \le t \le T_0} U_t
    = \frac47 n \ln \big((1\pm 2\eta)4n^{1/5}/7\big)+\chi n + O(n^{9/10}\ln n)
    =\frac{4}{35}n\ln n + \chi n \pm \epsilon n,
\]
and so
\[
    \sum_{0 \le t \le T_{n,M,\delta}} U_t = \frac{2}{7}n \ln n +\frac{4}{7}n \ln{\delta} +\chi n \pm 2\epsilon n.
\]
The result follows as $\epsilon > 0$ is arbitrary and $\delta$ can be chosen to be arbitrarily small.


\section{Proof of Theorem~\ref{thm:main} and Lemma~\ref{thm:processconvergence}  -- Dispersion Time}
\label{sec:mainProof}

Recall that $M=M(n)= n/2 + \alpha n^{1/2} + o(n^{1/2})$.
In what follows we fix a sufficiently small $\delta > 0$ and a sufficiently large~$n$ so that Lemma~\ref{lem:earlyStepsSummary} is applicable with $\epsilon = 1/2$. In particular,
\begin{equation}
\label{eq:TnMdeltaupper}
    T_{n,M,\delta} = \inf\{t > 0 : U_t \le n^{1/2}/\delta\}
    \le \delta n^{1/2}
    ~~\text{with probability at least}~~1-\delta.
\end{equation}
With this at hand set 
\begin{equation}
\label{eq:decompositiondispersiontime}
    T'_{n,M,\delta} 
    \coloneqq T_{n,M} - T_{n,M,\delta}.
\end{equation}
In words, $T^{\prime}_{n,M,\delta}$ is the number of steps that are taken by the dispersion process after the first step at which the number of unhappy particles drops below $n^{1/2}/\delta$. The forthcoming argument  goes roughly as follows. First, we will use the diffusion approximation toolbox from Section~\ref{ssec:preliminariesdiffusion} to establish that the number of unhappy particles after step
$T_{n,M,\delta}$, scaled appropriately, converges weakly to a process that satisfies \eqref{eq:SDEFellerDiffusionX}. From this, Lemma~\ref{thm:processconvergence} will follow quickly.
As a further consequence we will then deduce that~$n^{-1/2}T^{\prime}_{n,M,\delta}$ converges in distribution to the absorption time of this limiting process.
Finally, by combining~\eqref{eq:TnMdeltaupper} with Lemma~\ref{lem:infiniteinitialvalue} to take the limit $\delta \to 0$, we will conclude that $n^{-1/2}T_{n,M}$ converges to an almost surely positive random variable $T_{\alpha}$.
It will turn out that~$T_{\alpha}$ is distributed as the first time~$T$ at which the coordinate process hits zero under the measure $\mathbb{P}_{2\alpha,7/4,1,\infty}$ given in Lemma~\ref{lem:infiniteinitialvalue}, i.e.,
\begin{equation}
\label{eq:distributionTalpha}
    \mathbb{P}(T_{\alpha}\le s) = \mathbb{P}_{2\alpha, 7/4,1,\infty}(T\le s), \quad s \ge 0,
\end{equation}
and then Theorem~\ref{thm:main} will be established as well.
In order to make this outline precise we consider the (continuous) time-shifted process 
\[
    U'_{s} \coloneqq U_{\lfloor s \rfloor+T_{n, M, \delta}},
    \quad
    s \ge 0.
\]
By applying Theorem~\ref{thm:diffusionapproximation} we will show that $(n^{-1/2}U'_{sn^{1/2}})_{s\geq 0}$ converges weakly to a diffusion. Note that the following lemma is just a reformulation of  Lemma~\ref{thm:processconvergence} in the Introduction, as~\eqref{eq:SDEFellerDiffusion2} corresponds to the SDE \eqref{eq:SDEFellerDiffusionX}. 
\begin{lem}
\label{lem:diffusionapprox}
Let $\delta > 0$. As $n\to\infty$, the process $(n^{-1/2}U^{\prime}_{sn^{1/2}})_{s\geq 0}$ converges weakly to a process $X$ that satisfies
\begin{equation}
    \label{eq:SDEFellerDiffusion2}
    dX_{s} = \left(2\alpha X_{s} -\frac{7}{4}X_{s}^{2}\right) ds + \sqrt{X_{s}}dB_{s},~~s>0,
    \quad
    \text{and}
    \quad
    X_{0}=\delta^{-1}.
\end{equation}
\end{lem}
\begin{proof}
We will apply Theorem~\ref{thm:diffusionapproximation} with $h=h(n)=n^{-1/2}$ and $Y^{(n)}_{t}\coloneqq n^{-1/2}U'_t$ for $t\in\mathbb{N}_0$. First, note that it is necessary to extend the SDE \eqref{eq:SDEFellerDiffusion2} in a way that it has  a unique solution not only for all initial values $x\geq0$, but for all $x\in\mathbb{R}$. To this end, write $a^{+} = \max\{a,0\}$ for $a\in\mathbb{R}$ and consider the SDE
\begin{equation}
\label{eq:extendedFellerDiffusion}
    dX_{s} = \left(2\alpha X_{s}^{+} -\frac{7}{4}(X_{s}^{+})^{2}\right)ds + \sqrt{X_{s}^{+}}dB_{s},\quad  s>0, \text{ with } X_{0}=x\in\mathbb{R}.
\end{equation}
Note that if the initial value $x$ is negative, then $X=x$ uniquely satisfies this SDE. For ${x\geq0}$, recall that Corollary~\ref{cor:existenceanduniqueness} guarantees the existence of a unique solution $(X,B_{2\alpha, 7/4,1,x},\mathscr{P}_{2\alpha,7/4,1,x})$ to \eqref{eq:SDEFellerDiffusion2} with $X_{0} = x$ and such that $X\geq 0$ almost surely. Hence, if $x\ge 0$, \eqref{eq:extendedFellerDiffusion} coincides with \eqref{eq:SDEFellerDiffusion2} with initial value $X_{0}=x$ and we conclude that \eqref{eq:extendedFellerDiffusion} possesses a unique solution for all $x\in\mathbb{R}$.

Next, we employ Lemmas~\ref{lem:exponestepchange} and~\ref{asymptsecoment} with $\varepsilon(n) = 2\alpha n^{-1/2} + o(n^{-1/2})$, as  $M= n/2+\alpha n^{1/2}+o(n^{1/2})$. For this purpose, let $R<\infty$ and consider ${x\in S^{(n)} \subseteq \{0,n^{-1/2}, 2n^{-1/2},...,n^{1/2}\}}$ 
with $|x|\leq R$. Then, Lemma~\ref{lem:exponestepchange} with $u=xn^{1/2}$ implies that 
\[
    b^{(n)}(x) =  \frac{\mathbb{E}\left[n^{-1/2}U'_{t+1} - n^{-1/2}U'_{t} \mid n^{-1/2}U'_{t} = x\right]}{n^{-1/2}} =  2\alpha x - \frac{7}{4}x^{2} + o\left(R + R^3\right).
\]
Further, as $xn^{1/2} =o(n^{2/3})$ due to $|x|\le R$,  it follows from Lemma~\ref{asymptsecoment} with $u=xn^{1/2}$ that
\begin{equation}
    \label{eq:a^(n)(x)}
    a^{(n)}(x) = \frac{\mathbb{E}\left[(n^{-1/2}U'_{t+1} - n^{-1/2}U'_{t})^{2}\mid n^{-1/2}U'_{t} = x\right]}{n^{-1/2}}= x + o\left(R+R^2\right).
\end{equation}
We therefore obtain that for any $R<\infty$
\begin{equation}
\label{eq:conditionba}
    \lim_{n\to\infty} \sup_{x\in S^{(n)},|x|\leq R}\left|b^{(n)}(x)- \left(2\alpha x-\frac{7}{4}x^{2}\right)\right|=0 \text{ and } \lim_{n\to\infty} \sup_{x\in S^{(n)},|x|\leq R}|a^{(n)}(x)- x|=0.
\end{equation}
By Lemma~\ref{lem:earlyStepsSummary} we have
\begin{equation*}
\label{eq:UTnMdelta}
    U'_0 = U_{T_{n,M,\delta}} \sim n^{1/2}/\delta
    ~~\text{with probability}~~ 1-o(1).
\end{equation*}
Moreover, we will argue in a moment that
\[
    \gamma_{3}^{(n)}(x)
    = \frac{\mathbb{E}\left[|n^{-1/2}U'_{t+1} - n^{-1/2}U'_{t}|^{3}\mid n^{-1/2}U'_{t} = x\right]}{n^{-1/2}}
    = \frac{\mathbb{E}\left[|U'_{t+1} - U'_{t}|^{3}\mid U'_{t} = xn^{1/2}\right]}{n}
\]
satisfies  
\begin{equation}
\label{eq:conditiongamma}
    \lim_{n \to \infty} \sup_{x\in S^{(n)},|x|\leq R}|\gamma^{(n)}_{3}(x)| =0.
\end{equation}
The last two facts, together with~\eqref{eq:conditionba} and the existence of a unique solution to \eqref{eq:extendedFellerDiffusion} guarantee that we can apply Theorem~\ref{thm:diffusionapproximation} to conclude that $(n^{-1/2}U^{\prime}_{sn^{1/2}})_{s\geq 0}$ converges weakly to a process~$X$ that satisfies \eqref{eq:SDEFellerDiffusion2} with~$X_{0}=1/\delta$, and the proof is finished.

It remains to establish~\eqref{eq:conditiongamma}. Set 
\[
    \mathcal{E}_{t}
    \coloneqq \left \{ \left |U'_{t+1}-\mathbb{E}\left[U'_{t+1}\mid U'_{t}\right] \right|\leq 4\sqrt{U'_{t}\ln n} \right \}
\]
and start by noticing that $-U'_{t}\le U'_{t+1}-U'_{t}\le 2U'_{t} - U'_{t} = U'_{t}$, i.e. ${|U'_{t+1}-U'_{t}|\le U'_{t}}$. 
Then, according to Lemma~\ref{lem:roughcontrolfulltrajectory}, $\mathbb{P}(\mathcal{E}_{t}) = 1 -o(1/n)$ and we obtain     
    \begin{align*}
    \mathbb{E}\left[ \mathbb{1}_{\overline{\mathcal{E}_{t}}} \left |U'_{t+1}-U'_t \right |^{3} \mid U'_t= x n^{1/2}\right]
    \leq x^{3}n^{3/2}\mathbb{E}\left[\mathbb{1}_{\overline{\mathcal{E}_{t}}} \mid U'_t=xn^{1/2}\right]
    = o\left(R^{3}n^{1/2}\right).
    \end{align*}
    Additionally, on the event $\mathcal{E}_{t}$, we can use Lemma~\ref{lem:driftcoarsebounds} to obtain 
    \[
    |U'_{t+1}-U'_{t}|\le 4 \sqrt{U'_{t} \ln n} + |\varepsilon|U'_{t} + O\left((U'_{t})^{2} n^{-1}\right).  \]
    For $U'_{t}=x n^{1/2}$ and since $|\varepsilon| = O\left(n^{-1/2}\right)$, it follows that for $n$ sufficiently large
    \[
    |U'_{t+1}-U'_{t}|\le 4 \sqrt{x n^{1/2} \ln n} + O(x+x^2) \le 5\sqrt{R n^{1/2} \ln n}.
    \]
Combined with \eqref{eq:a^(n)(x)}, we obtain 
\begin{align*}
\mathbb{E}\left[ \mathbb{1}_{\mathcal{E}_{t}} \left |U'_{t+1}-U'_t \right |\left(U'_{t+1}-U'_t\right)^{2} \mid U'_t= x n^{1/2}\right]
& \leq 5\sqrt{Rn^{1/2}\ln n}\;\mathbb{E}\left[(U'_{t+1} -U'_{t})^{2} \mid U'_{t} = xn^{1/2}\right] \\
&= O\left(\left(R^{3/2} + R^{5/2}\right)n^{3/4}\sqrt{\ln n}\right).
\end{align*}
Hence, $\gamma_{3}^{(n)}(x) = O\big((R^{3/2} + R^{5/2})n^{-1/4}\sqrt{\ln{n}}+R^3 n^{-1/2}\big)$, which implies \eqref{eq:conditiongamma}. 
\end{proof}

Recall from Corollary~\ref{cor:existenceanduniqueness} that $(X,B_{2\alpha,7/4,1,x},\mathscr{P}_{2\alpha,7/4,1,x})$ represents a solution of~\eqref{eq:SDEFellerDiffusion2} with initial value $x\geq0$ and that the corresponding hitting time of zero is given by 
\begin{equation}\label{eq:limitstoppingtime}
    T = \inf\{s\ge 0 : X_{s}=0\}
\end{equation}
under the probability measure $\mathbb{P}_{2\alpha,7/4,1,x}$. The next statement asserts that $n^{-1/2}T^{\prime}_{n,M,\delta}$ converges in distribution to $T$ under $\mathbb{P}_{2\alpha,7/4,1,1/\delta}$.
\begin{lem}
\label{lem:convergenceTprime}
Let $\delta > 0$. Then, as $n \to \infty$, $n^{-1/2}T^{\prime}_{n,M,\delta}  \stackrel{d}{\to} T$, with $T$ given by \eqref{eq:limitstoppingtime} under the probability measure $\mathbb{P}_{2\alpha,7/4,1,1/\delta}$.
\end{lem}
\begin{proof}
Let $s\geq0$ and recall from Section~\ref{ssec:preliminariesdiffusion} that, under $\mathbb{P}_{2\alpha,7/4,1,1/\delta}$, $X_{s}=0$ is equivalent to $T\leq s$.
Similarly, $U_{\lfloor s\rfloor}=0$ if and only if $T_{n,M}\leq s$, from which we obtain that $T^{\prime}_{n,M,\delta} = T_{n,M}-T_{n,M,\delta}\leq s$ if and only if $U'_{s}=U_{\lfloor s \rfloor +T_{n,M,\delta}}=0$. Hence, 
\[
    \mathbb{P}_{2\alpha,7/4,1,1/\delta}(X_{s}=0) = \mathbb{P}_{2\alpha,7/4,1,1/\delta}(T\leq s)
    \quad\text{and}\quad
    \mathbb{P}(U^{\prime}_{s}=0) = \mathbb{P}(T^{\prime}_{n,M,\delta}\leq s).
\]
As  Lemma~\ref{lem:diffusionapprox} entails $\lim_{n\to\infty}\mathbb{P}\big(n^{-1/2}U^{\prime}_{sn^{1/2}}  = 0\big) = \mathbb{P}_{2\alpha,7/4,1,1/\delta}(X_s = 0)$,
we conclude \[\lim_{n\to\infty}\mathbb{P}(n^{-1/2}T^{\prime}_{n,M,\delta}\leq s) = \mathbb{P}_{2\alpha,7/4,1,1/\delta}(T\leq s). \qedhere \]
\end{proof}
With the convergence in distribution shown in the previous lemma at hand, we are now in the position to prove Theorem~\ref{thm:main}. Indeed, recall from~\eqref{eq:TnMdeltaupper} that $T_{n, M,\delta} \le \delta n^{1/2}$ with probability at least $1-\delta$ for $\delta$ sufficiently small and $n$ large enough. As $T_{n,M,\delta}$ is non-negative, this implies that
\[
    n^{-1/2}T^{\prime}_{n,M,\delta}
    \le n^{-1/2}T_{n,M}
    \le \delta +n^{-1/2}T^{\prime}_{n,M,\delta} \text{ with probability at least } 1-\delta.
\]
Applying Lemma~\ref{lem:convergenceTprime}, we obtain for all $s\geq 0$
\begin{equation*}
    \lim_{n \to \infty}\mathbb{P}(n^{-1/2}T_{n,M}\leq s)
    \le \lim_{n \to \infty}\mathbb{P}(n^{-1/2}T^{\prime}_{n,M,\delta} \le s) = \mathbb{P}_{2\alpha,7/4,1,1/\delta}(T \le s) 
\end{equation*}
and 
\begin{align*}
    \lim_{n \to \infty}\mathbb{P}\big(n^{-1/2}T_{n,M}\geq s\big)
    &\le \lim_{n \to \infty}\mathbb{P}\big(n^{-1/2}T_{n,M} \ge s,  T_{n,M,\delta} \le \delta n^{1/2}\big) + \delta \\
    &\le \mathbb{P}_{2\alpha,7/4,1,1/\delta}\big(T \ge s-\delta\big) + \delta.
\end{align*}
Note that for all $\mathbb{P}_{2\alpha,7/4,1,x_1}(T \ge \tau) \ge \mathbb{P}_{2\alpha,7/4,1,x_2}(T \ge \tau)$ for all 
$x_1 > x_2 \ge 0$ and $\tau \ge 0$, since~$X$ is almost surely continuous and needs a positive and finite amount of time to drop from~$x_1$ to~$x_2$. So, since according to Lemma~\ref{lem:infiniteinitialvalue} we have $\lim_{\delta \to 0}\mathbb{P}_{2\alpha,7/4,1,1/\delta} = \mathbb{P}_{2\alpha,7/4,1,\infty}$ and $T$ is continuous, 
\[
    \lim_{\delta \to 0}\mathbb{P}_{2\alpha,7/4,1,1/\delta}\big(T \ge s-\delta\big)\le \lim_{\delta \to 0}\mathbb{P}_{2\alpha,7/4,1,\infty}\big(T \ge s-\delta\big) =  \mathbb{P}_{2\alpha,7/4,1,\infty}\big(T \ge s\big),
\]
which yields
\[
    \lim_{n \to \infty}\mathbb{P}(n^{-1/2}T_{n,M}\leq s)
    = \mathbb{P}_{2\alpha,7/4,1,\infty}(T\leq s),
    \quad s \ge 0.
\]
Thus, $n^{-1/2}T_{n,M} \stackrel{d}{\to} T_{\alpha}$, where $T_\alpha$ satisfies \eqref{eq:distributionTalpha}.
Moreover, $\mathbb{P}_{2\alpha, 7/4,1,\infty}(T>0) = 1$ implies that~$T_{\alpha}$ is positive almost surely, and this completes the proof of Theorem \ref{thm:main}.


\section{Absorption Time of the Diffusion Process}
\label{sec:absorptiontime}

Let $\alpha \in \mathbb{R}$. In this section we study properties of the absorption time of the diffusion process $T_{\alpha}$ in Theorem~\ref{thm:main}, which, as we showed in Section~\ref{sec:mainProof}, is distributed like the first time $T$ the coordinate process hits zero under the measure $\mathbb{P}_{2\alpha,7/4,1,\infty}$, see also~\eqref{eq:distributionTalpha}.
We begin with considering the expected value
\[
    \lim_{n\to \infty}\mathbb{E}\left[n^{-1/2}T_{n,M}\right]
    = \mathbb{E}\left[T_{\alpha}\right]
    = \mathbb{E}_{2\alpha,7/4,1,\infty}\left[T\right],
\]
where $M = n/2 + \alpha n^{1/2} + o(n^{1/2})$, and provide asymptotic bounds when $|\alpha|\to\infty$. 
\begin{lem}
\label{seriesrepr}
Let $\alpha \in \mathbb{R}$. Then 
\begin{equation}\label{asymean}
    \mathbb{E}\left[T_{\alpha}\right]
    = \frac{\pi^{3/2}}{\sqrt{7}}+   \frac1{\sqrt{7}}\sum_{m \ge 1} \frac{\Gamma(\frac{m+1}{2})}{m!}\left(\frac{8\alpha}{\sqrt{7}}\right)^m t_m,
\end{equation}
where
\[
    t_m
    \coloneqq \sum_{k \ge 0} \frac2{(\frac{m+1}{2} + 2k)(\frac{m+3}{2} + 2k)}
    = H_{(m-1)/4} - H_{(m-3)/4},
    \quad m \in \mathbb{N}_0,
\]
and $H_x = \sum_{k \ge 1} \Big(\frac1k - \frac1{k+x}\Big) $ for all $x \in\mathbb{R}$ with $x>-1$.
\end{lem}
\begin{proof}
Abbreviate $A = A(\alpha)\coloneqq{8\alpha}/{7}$.
From \eqref{eq:theta} and Lemma~\ref{lem:infiniteinitialvalue}, we obtain that
\begin{equation*}
\label{eq:integralrepETa}
    \frac{7}{4}\mathbb{E}\left[T_{\alpha}\right]= \int_{0}^{\infty}\frac{1}{\lambda}\exp\left(-\frac{1}{7}\lambda^{2}+A\lambda\right)\int_{0}^{\lambda}\exp\left(\frac{1}{7}\mu^{2} - A\mu\right) d\mu\;d\lambda \enspace.
\end{equation*}
Substituting $\mu = \lambda u$ and applying Fubini's theorem yields 
\begin{equation}
\label{eq:Eamain}
    \frac{7}{4}\mathbb{E}\left[T_{\alpha}\right]
    = \int_0^1 \int_{0}^{\infty}\exp\left(-\frac{1}{7}(1-u^2)\lambda^2 + A(1-u)\lambda\right) d\lambda \; du.
\end{equation}
We shall use the well-known fact that for any $c > 0$ and $n > -1$
\[
    \int_{0}^{\infty}x^{n}\exp(-c x^{2})dx = \frac{1}{2} c^{-(n+1)/2}\Gamma\left(\frac{n+1}{2}\right),
\]
which follows directly by a simple change of variables in the definition of the Gamma function.
By expanding $\exp(A(1-u)\lambda) = \sum_{m \ge 0} (A(1-u))^m\lambda^m/m!$ we obtain, again by Fubini's theorem,
\[
\begin{split}
    \int_0^\infty\exp\left(-\frac{1}{7}(1-u^2)\lambda^2 +  A(1-u)  \lambda\right) & d\lambda
    = \int_0^\infty\exp\left(-\frac{1}{7}(1-u^2)\lambda^2\right) \sum_{m \ge 0} \frac1{m!}(A(1-u))^m\lambda^m d\lambda \\
    & = \sum_{m \ge 0} \frac1{m!}(A(1-u))^m \int_0^\infty \lambda^m\exp\left(-\frac{1}{7}(1-u^2)\lambda^2\right)  d\lambda \\
    & = \frac12 \sum_{m \ge 0} \frac1{m!}(A(1-u))^m \cdot \Gamma\left(\frac{m+1}2\right) \left(\frac{1-u^2}{7}\right)^{-(m+1)/2}.
\end{split}
\]
Since $1-u^2 = (1-u)(1+u)$ we further obtain from~\eqref{eq:Eamain} and Fubini that
\begin{equation}
\label{eq:Eamain2}
\begin{split}
    \frac74\mathbb{E}\left[T_{\alpha}\right]
    & =  \frac{\sqrt{7}}{2} \int_0^1 \sum_{m\ge 0} \frac{\Gamma(\frac{m+1}2)}{m!} \left(\sqrt{7}A\right)^m \frac{(1-u)^{(m-1)/2}}{(1+u)^{(m+1)/2}}du \\
    & =  \frac{\sqrt{7}}{2} \sum_{m\ge 0} \frac{\Gamma(\frac{m+1}2)}{m!} \left(\frac{8\alpha}{\sqrt7}\right)^m \int_0^1 \frac{(1-u)^{(m-1)/2}}{(1+u)^{(m+1)/2}} du.   
\end{split}
\end{equation}
To evaluate the last integral, we substitute $u = (1-\sqrt{v})/(1+\sqrt{v})$, i.e.~$du = -dv/(\sqrt{v}(1+\sqrt{v})^2)$. Moreover, let $x > -1$, and notice that by integration by parts, the fact that $(1-y^x)\ln(1-y) \to 0$ as $y\to 1$, and the series representation of the logarithm 
\begin{equation}
\label{eq:integralrepresentationHx}
    \int_{0}^{1} \frac{1-y^{x}}{1-y}dy
    = -\int_{0}^{1}xy^{x-1}\ln(1-y)dy
    = \sum_{k\ge 1} x \int_{0}^{1}\frac{y^{x+k-1}}{k}dy
    = \sum_{k\ge 1} \frac{x}{k (x+k)} = H_{x}.
\end{equation}
Thus
\[
    \int_0^1 \frac{(1-u)^{(m-1)/2}}{(1+u)^{(m+1)/2}} du
    = \int_{0}^{1}\frac{v^{(m-3)/4}}{2(1+\sqrt{v})}dv
    = \int_{0}^{1}\frac{v^{(m-3)/4} - v^{(m-1)/4}}{2(1-v)}dv
    = \frac{H_{(m-1)/4}- H_{(m-3)/4}}{2}.
\]
By plugging this expression into~\eqref{eq:Eamain2}, we obtain
 \[ 
    \mathbb{E}\left[T_{\alpha}\right]
    =\frac2{\sqrt{7}} \Gamma\Big(\frac{1}{2}\Big)\int_0^1 \frac{1}{\sqrt{1-u^2}}du + \frac1{\sqrt{7}}\sum_{m \ge 1} \frac{\Gamma(\frac{m+1}{2})}{m!}\left(\frac{8\alpha}{\sqrt{7}}\right)^m \left(H_{(m-1)/4} - H_{(m-3)/4} \right). 
\]
The statement now follows from $\Gamma(1/2) = \sqrt{\pi}$ and $ \int_0^1 1/\sqrt{1-u^{2}}du = \pi/2.$ 
\end{proof}

Next we establish the asymptotic behaviour for $|\alpha| \to \infty$ of $\mathbb{E}[T_{\alpha}]$. In order to do so, we do not use the explicit sum from~\eqref{asymean} -- although we could do so --  but rather exploit the integral representation that we have encountered previously. In particular, note that from Lemma~\ref{lem:infiniteinitialvalue}, see also~\eqref{eq:Eamain}, we obtain
\begin{equation*}
    \frac{7}{4}\mathbb{E}\left[T_{\alpha}\right]
    = \int_0^1 \int_{0}^{\infty}\exp\Big(-\frac17(1-u^2)\lambda^2 + \frac{8\alpha}{7}(1-u)\lambda\Big) d\lambda \; du.
\end{equation*}
By substituting $u$ for $1-u$ we obtain
\begin{equation}
\label{eq:Etaintrexpr}
    \frac{7}{4}\mathbb{E}\left[T_{\alpha}\right]
    = \int_0^1 \int_{0}^{\infty}\exp\big(-Q\lambda^2 + L\lambda\big) d\lambda \; du,
    \quad \text{where} \quad
    Q = \frac{2u-u^2}{7}, ~ L = \frac{8\alpha u}{7},
\end{equation}
which is the starting point in the following two proofs. We continue with the case $\alpha \to \infty$, where we show that $\mathbb{E}\left[T_{\alpha}\right]$ grows very fast, namely quadratic exponential in $\alpha$. 
\begin{cor}
\label{cor:ato+infty}
As $\alpha \rightarrow \infty$, we have
\begin{equation}
\label{eq:ETa->infty}
    \mathbb{E}\left[T_{\alpha}\right]
    = (1+o(1)) \frac{\sqrt{7\pi}}8 \cdot \frac{e^{16\alpha^2/7}}{\alpha^2}.
\end{equation}
\end{cor}
\begin{proof}
In order to estimate~\eqref{eq:Etaintrexpr} we first study in more abstract terms the general behavior of the integral of $\exp(-A\lambda^2 + B\lambda)$ where $A,B$ are parameters; we will be particularly interested in the case where $A$ is positive and bounded (later we will set $A = Q = \frac17(2u-u^2)$ for $u\in[0,1]$) and~$B \to \infty$ (later we will set $B = L = 8\alpha u/7$, and the main contribution will come from $u$ close to 1). Note that $-A\lambda^2 + B\lambda = -(\sqrt{A}\lambda - {B}/{2\sqrt{A}})^2 + {B^2}/{4A}$, so by changing the variable $x=\sqrt{A}\lambda - {B}/{2\sqrt{A}}$, we obtain 
\begin{equation}
\label{eq:intStart}    
    \int_0^\infty \exp\big( -A\lambda^2 + B\lambda\big) d\lambda
    =  \frac{1}{\sqrt{A}} e^{B^2/4A} \int_{-B/2\sqrt{A}}^\infty e^{-x^2} dx.
\end{equation}
Since $\int_{-\infty}^\infty e^{-x^2}dx = \sqrt{\pi}$ we readily obtain the general upper bound
\begin{equation}
\label{eq:generalupperERF}    
    \int_0^\infty \exp\big(-A\lambda^2 + B\lambda\big) d\lambda
    \le \sqrt{\frac{\pi}A} e^{B^2/4A}, \quad A,B > 0.
\end{equation}
However, we will also require tighter bounds. Note that
\[
    \int_{-\infty}^{-B/2\sqrt{A}} e^{-x^2} dx=
    \int_{-\infty}^{-B/2\sqrt{A}} \frac{|x|}{|x|}e^{-x^2} dx
    \le \frac{2\sqrt{A}}{B} \int_{-\infty}^{-B/2\sqrt{A}}  |x| e^{-x^2} dx
    = \frac{\sqrt{A}}{B}e^{-B^2/4A}
\]
so that by 
\[
    \int_{-B/2\sqrt{A}}^\infty e^{-x^2} dx
    = \sqrt{\pi} - \int_{-\infty}^{-B/2\sqrt{A}} e^{-x^2} dx
\]
we obtain from~\eqref{eq:intStart} that
\begin{equation}
\label{eq:generalERF}    
    \int_0^\infty \exp\big(-A\lambda^2 + B\lambda\big) d\lambda
    = (\pi/A)^{1/2} e^{B^2/4A} 
    \pm \frac1B,
    \quad A,B > 0.
\end{equation}
With these preparations at hand we turn to the task of estimating~\eqref{eq:Etaintrexpr}. Let $0 < \delta < 1/2$ be some (small) quantity to be specified later. We split the integration range $u\in(0,1]$ (neglecting $u=0$) in three parts $(0,1/2], [1/2,1-\delta]$ and $[1-\delta,1]$. Let us consider the interval $[1-\delta,1]$ first, namely we start by estimating
\begin{equation}\label{mainrange}
    \int_{1-\delta}^{1}\int_0^\infty \exp\big(-A\lambda^2 + B\lambda\big) d\lambda du.
\end{equation}
As we will see, the main contribution to~\eqref{eq:Etaintrexpr} comes from~\eqref{mainrange}. We will apply~\eqref{eq:generalERF} with 
\[
    A = Q = \frac{2u-u^2}{7},
    B = L = \frac{8\alpha u}{7},
    \quad\text{so that}\quad
    \frac{B^2}{4A} = \frac{16\alpha^2}{7}\frac{u}{2-u}.
\]
Note that for this particular choice of $A$ and $B$ we obtain
\begin{equation}\label{int1}
    \int_{1-\delta}^{1}\int_0^\infty \exp\big(-A\lambda^2 + B\lambda\big) d\lambda du=\int_{1-\delta}^{1}\left(\frac{\sqrt{7\pi}}{\sqrt{2u-u^2}}e^{\frac{16\alpha^2}{7}\frac{u}{2-u}}\pm \frac{7}{8\alpha u}\right)du.
\end{equation}
With the purpose of simplifying the terms $u/(2-u)$ and $1/\sqrt{2u-u^2}$ observe that
\begin{equation}\label{RR}
    \frac{u}{2-u} - (1 - 2(1-u))=\frac{2(1-u)^2}{2-u}
\end{equation}
and, since $1-\delta\leq u\leq 1$, we readily see that the ratio on the right-hand side of (\ref{RR}) is at most $2(1-u)^2\leq 2\delta^2$. Whence we arrive at
\[
    0\leq \frac{u}{2-u} - (1 - 2(1-u)) \le  2\delta^2.
\]
Moreover, for $1-\delta\leq u\leq 1$ we also have that
\[\left|\frac{1}{\sqrt{2u - u^2}} - 1\right| \overset{}{\leq} 1-u \le \delta,\]
since, obviously, $\sqrt{2u-u^2} \le 1$ and moreover, by writing, say, $u = 1-y$ for $0 \le y \le 1/2$ and simplifying, $\sqrt{2u-u^2}(2-u) \ge 1$.
Going back to (\ref{int1}) we see that, uniformly in $\delta$ as $\alpha\to\infty$
\[
    \int_{1-\delta}^1 \int_0^\infty \exp\big(-A\lambda^2 + B\lambda\big) d\lambda du
    = (1 + O(\delta)) \int_{1-\delta}^1 \sqrt{7\pi}  e^{\frac{16\alpha^2}{7}(1- 2(1-u) + O(\delta^2))}du.
\]
Since 
\[
    \int_{1-\delta}^{1}e^{-\frac{32 \alpha^2}{7}(1-u)}du=\frac{7}{32 \alpha^2}\left(1-e^{-\frac{32 \alpha^2}{7}\delta}\right)
\]
we arrive at the expression
\begin{equation}\label{finexpress}
    \int_{1-\delta}^1 \int_0^\infty \exp\big(-A\lambda^2 + B\lambda\big) d\lambda du
    = (1 + O(\delta))  \sqrt{7\pi}  e^{\frac{16\alpha^2}{7}(1 + O(\delta^2))} \frac{7}{32\alpha^2} (1-e^{-32\alpha^2\delta/7}).
\end{equation}
We now choose the value of $\delta$ that will be most convenient to work with. Fix $\delta = 9\alpha^{-2}\ln\alpha$, so that $\alpha^2\delta \to \infty$ and $\alpha^2\delta^2 \to 0$ as $\alpha\to\infty$. Substituting this value of $\delta$ into \eqref{finexpress} we conclude that
\begin{equation}
\label{lastint}
    \int_{1-\delta}^1 \int_0^\infty \exp\big(-A\lambda^2 + B\lambda\big) d\lambda du
    \sim \frac{(7^3\pi)^{1/2}}{32} \alpha^{-2} e^{\frac{16\alpha^2}{7}},
    \quad
    \alpha \to \infty
    .
\end{equation}
Let us now turn our attention to the range $u \in (0,1/2] \cup [1/2, 1-\delta]$ in~\eqref{eq:Etaintrexpr}.
For the first interval we apply~\eqref{eq:generalupperERF} with $A = Q = \frac{2u-u^2}{7}, B = L = \frac{8\alpha u}{7}$, so that
\[
    \frac{B^2}{4A} = \frac{16\alpha^2}{7}\frac{u}{2-u} \le \frac13 \frac{16\alpha^2}{7},
    \quad u > 0.
\]
Thus, recalling the basic integral $\int 1/\sqrt{2u-u^2} du = \text{arcsin}(u-1) + C$ and the values $\text{arcsin}(-1) = -\pi/2$ and $\text{arcsin}(-1/2) = -\pi/6$, 
we arrive (by~\eqref{eq:generalupperERF}) at the (rough) bound
\begin{equation}
\label{eq:ETa0..1/2}
    \int_0^{1/2} \int_0^\infty \exp\big(-A\lambda^2 + B\lambda\big) d\lambda du
    \le e^{\frac13 \frac{16\alpha^2}{7}} \int_0^{1/2} \frac{\sqrt{7\pi}}{\sqrt{2u-u^2}}du
    = \frac{\sqrt{7}\pi^{3/2}}3 e^{\frac13 \frac{16\alpha^2}7},
\end{equation}
which is much smaller than the right-hand side in \eqref{lastint}. We conclude by considering the case $u\in[1/2,1-\delta]$. We apply again~\eqref{eq:generalupperERF} with $A = Q = \frac{2u-u^2}{7}, B = L = \frac{8\alpha u}{7}$, so that 
\[
    \frac{B^2}{4A}\le (1-\delta)\frac{16\alpha^2}{7}
    \quad
    \text{and, say,}
    \quad
    \frac{1}{\sqrt{A}} \le 4.
\]
Thus, we obtain
\begin{equation*}
\label{eq:ETa1/2..1-delta}
    \int_{1/2}^{1-\delta} \int_0^\infty \exp\big(-A\lambda^2 + B\lambda\big) d\lambda du
    \le 4\sqrt{\pi}\int_{1/2}^{1-\delta} e^{(1-\delta)\frac{16\alpha^2}{7}} du
    \le 4\sqrt{\pi} e^{(1-\delta)\frac{16\alpha^2}{7}},
\end{equation*}
which, by our choice of $\delta = 9\alpha^{-2}\ln\alpha$, is at most $4\sqrt{\pi}\alpha^{-9}e^{{16\alpha^2}/7}$. By combining this with~\eqref{eq:ETa0..1/2} for the case $u\in(0,1/2]$ and~\eqref{lastint} for $u\in[1-\delta,1]$ we finally obtain, as $\alpha\to\infty$, from~\eqref{eq:Etaintrexpr} the claim.
\end{proof}
In contrast to the quadratic exponential behavior as $\alpha\to\infty$ established in Corollary~\ref{cor:ato+infty}, we show that the behavior is rather moderate (i.e., polynomial with logarithmic corrections) when~$\alpha \to -\infty$. 
\begin{cor}
\label{cor:ato-infty}
As $\alpha \rightarrow -\infty$, we have
\begin{equation*}
    \mathbb{E}\left[T_{\alpha}\right]
    = (1+o(1))\frac{\ln |\alpha|}{|\alpha|}.
\end{equation*}
\end{cor}
\begin{proof}
Let us start with two auxiliary estimates that we shall exploit. Using $1-e^{-y}\leq y$ for $y\geq 0$ we obtain $e^{-Qx^2} = 1 \pm Qx^2$ for all $x\in \mathbb{R}, Q>0$. Also, $\int_0^\infty e^{-Lx} dx = 1/L$ for all $L > 0$ and we obtain (for $Q,L > 0$)
\begin{equation}
\label{eq:intestimateLARGE}
    \left|\int_0^\infty \exp\Big(-Qx^2 - Lx\Big) dx - \frac1L\right| =\int_{0}^{\infty}e^{-Lx}dx - \frac{1}{L} \pm Q\int_{0}^{\infty}x^2e^{-Lx}dx
    \le 
    2Q L^{-3},
\end{equation}
where the inequality follows from integration by parts of the second integral.
This estimate will be useful when $L$ is large. On the other hand, when $L$ is small, we will use a different bound. Indeed, noting that $|e^{-Lx} -1| \le Lx $ for all $Lx > 0$ and that $\int_0^\infty e^{-Qx^2} dx = \sqrt{\pi}/2\sqrt{Q}$ and $\int_0^\infty xe^{-Qx^2} dx = 1/2Q$ for $Q > 0$, we obtain 
\begin{equation}
\label{eq:intestimateSMALL}
    \left|\int_0^\infty \exp\Big(-Qx^2 - Lx\Big) dx - \frac{\sqrt{\pi}}{2\sqrt{Q}}\right|
    \le \frac{L}{2Q},
    \quad
    Q,L > 0.
\end{equation}
With those facts at hand we continue with the proof of the corollary.
From~\eqref{eq:Etaintrexpr} we readily obtain 
\begin{equation*}
\label{eq:EtaintrexprII}
    \frac{7}{4}\mathbb{E}\left[T_{\alpha}\right]
    = \int_0^1 \int_{0}^{\infty}\exp\big(-Q\lambda^2 -L\lambda\big) d\lambda \; du,
    \quad \text{where} \quad
    Q = \frac{2u-u^2}{7}, ~ L = \frac{8|\alpha|u}{7}.
\end{equation*}
We will use~\eqref{eq:intestimateSMALL} to estimate the integral when~$u$ is close to zero and~\eqref{eq:intestimateLARGE} otherwise. More precisely, let $0 < \delta < 1$. Then, by~\eqref{eq:intestimateSMALL} together with the fact that for any $0 \le u \le 1$
\[
    2LQ^{-1}
    =\frac{16|\alpha|u}{2u-u^2}
    \leq 16|\alpha|
\]
we obtain
\begin{align*}
    \int_0^\delta \int_0^\infty \exp\big(-Q\lambda^2 -L\lambda\big)d\lambda \; du
    &= \int_0^\delta \left(\frac{\sqrt{\pi}}{2\sqrt{Q}}\pm 16|\alpha|\right)du
    =  \frac{\sqrt{7\pi}}{2}\int_{0}^{\delta}\frac{du}{\sqrt{2u-u^2}} \pm 16|\alpha| \delta.
\end{align*}
Since $u\leq 2u-u^2\leq 2u$ for $0 \le u \le \delta \le 1$, 
we obtain (uniformly in $\delta, \alpha$)
\begin{equation}
\label{eq:int0delta}
    \int_0^\delta \int_0^\infty \exp\big(-Q\lambda^2 -L\lambda\big)d\lambda \; du
    = O\big(\sqrt{\delta} + |\alpha| \delta\big).
\end{equation}
Moreover, by~\eqref{eq:intestimateLARGE} and $QL^{-3} \le |\alpha|^{-3}u^{-2}$ we obtain (again uniformly in $\delta,\alpha$) 
\[
\begin{split}
    \int_\delta^1 \int_0^\infty \exp\big(-Q\lambda^2 -L\lambda\big)d\lambda \; du
    & = \int_\delta^1 \frac1L + O(QL^{-3})du
    = \frac{7\ln(1/\delta)}{8|\alpha|} + O(\delta^{-1} |\alpha|^{-3}).
\end{split}
\]
By choosing $\delta = \alpha^{-2}$ and combining the last estimate with~\eqref{eq:int0delta} we obtain from~\eqref{eq:Etaintrexpr}, as $\alpha \to -\infty$,
\[
    \frac{7}{4}\mathbb{E}\left[T_{\alpha}\right]
    = \frac{7\ln(1/\delta)}{8|\alpha|} + O\big(\sqrt{\delta} + |\alpha|\delta + \delta^{-1}/|\alpha|^3\big)
    \sim \frac74\frac{\ln|\alpha|}{|\alpha|},
\]
and the claim follows.
\end{proof}

\section{Proof of Theorem~\ref{numbjumps} -- Total Number of Jumps}
\label{sec:proofnumjumps}
Recall that $M=M(n)= n/2 + \alpha n^{1/2} + o(n^{1/2})$ with $\alpha \in \mathbb{R}$. Our goal is to establish that there exists a continuous random variable $A_\alpha$ such that, as $n\to \infty$, 
\begin{equation}
    \label{eq:claim}
    n^{-1}\bigg(\sum_{t \ge 0} U_t - \frac27n\ln n\bigg)
    \overset{d}{\longrightarrow}
    A_\alpha.
\end{equation}
For this purpose we fix an $\epsilon > 0$ and $\delta > 0$ sufficiently small and~$n$  sufficiently large so that Lemma~\ref{lem:earlyStepsSummary} applies. In particular, there exists $\chi \in \mathbb{R}$ such that, with probability at least $1-\delta$, 
\begin{equation}
    \label{eq:earlySteps}
    U_{T_{n,M,\delta}} \sim \frac{\sqrt{n}}{\delta} \quad  \text{ and } \sum_{0 \le t \le T_{n,M,\delta}}U_{t} = \frac{2}{7}n\ln n + \frac{4}{7}n \ln \delta + \chi n \pm \epsilon n,
\end{equation}
 where $T_{n,M,\delta} = \inf \{t>0 : U_{t}\le n^{1/2}/\delta\}$. In the following lemma we consider all $t \ge T_{n,M,\delta}$. 
\begin{lem}
    \label{lem:convergenceSumToIntegral}
    Let $\delta >0$. As $n \to \infty$, 
    \[
    n^{-1}\sum_{t \ge T_{n,M,\delta}} U_t \overset{d}{\longrightarrow} \int_{0}^{\infty}X_{s}ds, 
    \]
     where $X$ satisfies \eqref{eq:SDEFellerDiffusionX} and thus depends on $\delta$ through the initial condition $X_{0}=\delta^{-1}$.
\end{lem}

\begin{proof}
We use the notation from Section~\ref{sec:mainProof} and consider the process of unhappy particles after $T_{n,M,\delta}$, i.e. $U'_s = U_{\lfloor s\rfloor + T_{n,M,\delta}}$ for all $s \ge 0$. Hence, 
\begin{equation}\label{eq:rewrite}
    n^{-1}\sum_{t \ge T_{n,M,\delta}} U_t = n^{-1}\sum_{t \ge 0} U'_t = n^{-1}\int_{0}^{\infty}U'_{s}ds= n^{-1/2}\int_{0}^{\infty}U'_{sn^{1/2}}ds. 
\end{equation}
Let $S >0$. For every $m\in \mathbb{N}$, let $\pi_{m}: 0 =s_{0} < s_{1} < ... < s_{m}=S$ be a partition of the interval $[0,S]$ such that $\lim_{m\to \infty}\sup_{i\in\{1,...,m\}}|s_{i}-s_{i-1}|=0$.
Recall Lemma~\ref{lem:diffusionapprox}, from which we obtain that the process $(n^{-1/2}U'_{sn^{1/2}})_{s\ge0}$ converges weakly to $X$ as $n\to \infty$ in the Skorokhod space $D([0,S],\mathbb{R})$. Since, for each$i \in \{1,...,m\}$, the mapping $F(f) = \sup_{s\in[s_{i-1},s_{i}]}f(s)$, $f \in D([0,S],\mathbb{R})$, is continuous with respect to the Skorokhod metric, this implies that as $n\to\infty$
\[
    n^{-1/2}\sum_{i=1}^m(s_i - s_{i-1})\sup_{s \in [s_{i-1},s_{i}]}U'_{sn^{1/2}}
    \overset{d}{\longrightarrow}
    \sum_{i=1}^m (s_{i}-s_{i-1}) \sup_{s \in [s_{i-1},s_{i}]}X_s.
\]
    Noting that 
    \[ \int_{0}^{S}U'_{sn^{1/2}}ds \le \sum_{i=1}^{m} (s_{i}-s_{i-1})\sup_{s\in [s_{i-1},s_{i}]} U'_{sn^{1/2}} \]
    we obtain, for all $a\ge 0$,
    \begin{align*}
    \lim_{n\to \infty}\mathbb{P}\left(n^{-1/2}\int_{0}^{S}U'_{sn^{1/2}}ds\le a\right) &\ge \lim_{n\to \infty}\mathbb{P}\left(n^{-1/2}\sum_{i=1}^{m}(s_{i}-s_{i-1})\sup_{s\in [s_{i-1},s_{i}]}U'_{sn^{1/2}}\le a\right)\\
    &= \mathbb{P}_{2\alpha,7/4,1,1/\delta}\left(\sum_{i=1}^{m}(s_{i}-s_{i-1})\sup_{s\in [s_{i-1},s_{i}]}X_{s}\le a\right).
    \end{align*} 
    Analogously, we obtain 
    \begin{align*}
    \lim_{n\to \infty}\mathbb{P}\left(n^{-1/2}\int_{0}^{S}U'_{sn^{1/2}}ds\le a\right) \le \mathbb{P}_{2\alpha,7/4,1,1/\delta}\left(\sum_{i=1}^{m}(s_{i}-s_{i-1})\inf_{s\in [s_{i-1},s_{i}]}X_{s}\le a\right).
    \end{align*}
    As $X$ has almost surely continuous paths, $\int_{0}^{S}X_{s}ds$ is a Riemann integral for almost every realisation and hence
     \[ 
     \lim_{m\to \infty}\sum_{i=1}^{m}(s_{i}-s_{i-1})\inf_{s \in [s_{i-1},s_{i}]}X_{s}
     = \int_0^S X_s ds
     = \lim_{m\to \infty}\sum_{i=1}^{m}(s_{i}-s_{i-1})\sup_{s \in [s_{i-1},s_{i}]}X_s.
     \]    
This yields the equality
\[
\lim_{n\to \infty}\mathbb{P}\left(n^{-1/2}\int_{0}^{S}U'_{sn^{1/2}}ds\le a\right) = \mathbb{P}_{2\alpha,7/4,1,1/\delta}\left(\int_{0}^{S}X_{s}ds\le a\right), \quad a\ge 0.
\]
As $U' \ge 0$, we readily obtain, for all $a\ge 0$,
\begin{equation}
\label{eq:ubsumut}
\begin{split}
    \lim_{n\to \infty}\mathbb{P}\left(n^{-1/2}\int_{0}^{\infty}U'_{sn^{1/2}}ds\le a\right) &\le \lim_{n\to \infty}\mathbb{P}\left(n^{-1/2}\int_{0}^{S}U'_{sn^{1/2}}ds\le a\right)\\
    &=\mathbb{P}_{2\alpha,7/4,1,1/\delta}\left(\int_{0}^{S}X_{s}ds\le a\right).
\end{split}
\end{equation}
 Additionally, since $U'_{s}=0$ for all $s\ge T'_{n,M,\delta}$ 
 we estimate that
    \[ \mathbb{P}\left(n^{-1/2}\int_{0}^{\infty}U'_{sn^{1/2}}ds\ge a\right) \le  \mathbb{P}\left(n^{-1/2}\int_{0}^{S}U'_{sn^{1/2}}ds\ge a\right) + \mathbb{P}\left( n^{-1/2}T'_{n,M,\delta}> S\right), \quad a \ge 0,\]
    from which we readily deduce using Lemma~\ref{lem:convergenceTprime} that, for all $a\ge 0$, 
     \[ \lim_{n\to \infty}\mathbb{P}\left(n^{-1/2}\int_{0}^{\infty}U'_{sn^{1/2}}ds\ge a\right) \le \mathbb{P}_{2\alpha,7/4,1,1/\delta}\left(\int_{0}^{S}X_{s}ds\ge a\right) + \mathbb{P}_{2\alpha,7/4,1,1/\delta}\left( T> S\right).\]
     The claim is established by letting $S \to \infty$,\eqref{eq:ubsumut} and \eqref{eq:rewrite}.
\end{proof}

Recall Lemma~\ref{lem:integralequalsTprime}, where we established that the integral $\int_{0}^{\infty}X_{s}ds$ from the previous statement equals the hitting time of zero of an Ornstein-Uhlenbeck process with suitable parameters depending on $\alpha \in \mathbb{R}$ and $\delta>0$. As the distribution of such a stopping time is well-known, see Lemma~\ref{lem:densityTprime}, this connection is the key element in the proof of the convergence in distribution of $\int_{0}^{\infty}X_s ds + \frac47\ln\delta$ as $\delta \to 0$, which is stated below. 
\begin{lem}
    \label{lem:convergenceIntegral}
    For all $\alpha \in \mathbb{R}$ there exists a continuous random variable $\tilde{A}_{\alpha}$ such that as $\delta \to 0$
    \[\int_{0}^{\infty}X_{s}ds + \frac{4}{7}\ln \delta \overset{d}{\longrightarrow} \tilde{A}_{\alpha},  \]
    where $X$ satisfies \eqref{eq:SDEFellerDiffusionX} and thus depends on $\delta$ through the initial condition $X_{0}=\delta^{-1}$. In particular, for $\alpha =0$,
    \[
    \mathbb{P}(\tilde{A}_{0} \ge a) =  \text{\emph{erf}}\left(\frac{\sqrt{7}}{2}e^{-7a/4} \right), \quad a \in \mathbb{R}.
    \]
\end{lem}

\begin{proof}
Let us begin with the case $\alpha =0$. Abbreviate $\mathbb{P}_{0,7/4,1,1/\delta}$ with $\mathbb{P}_{\delta}$. From Lemmas~\ref{lem:integralequalsTprime} and~\ref{lem:densityTprime} we obtain for all $a\ge 0$
\begin{align*}
    \mathbb{P}_{\delta} & \left(\int_{0}^{\infty}X_{s}ds \ge a\right) =\left(\frac{7}{2} \right)^{3/2}\frac{\delta^{-1}}{\sqrt{2\pi}} \int_{a}^{\infty}\left(e^{7s/2}-1\right)^{-3/2}e^{7s/2}\exp\left(-\frac{7\delta^{-2}}{4(e^{7s/2}-1)}\right)ds.
    \end{align*}
     The substitution $u = \sqrt{7}\delta^{-1}(e^{7s/2}-1)^{-1/2}/2$ leads to
     \begin{align}
     \label{eq:distribution Integral}
         \mathbb{P}_{\delta}\left(\int_{0}^{\infty}X_{s}ds \ge a\right)&= \frac{2}{\sqrt{\pi}}\int_0^{\sqrt7\delta^{-1}(e^{7a/2}-1)^{-1/2}/2} \exp(-u^2) du 
         = \text{erf}\left( \frac{\sqrt{7}\delta^{-1}}{2\sqrt{e^{7a/2}-1}}\right).
     \end{align}
    As for all $a \in \mathbb{R}$ we have $a - \frac47\ln\delta \ge 0$ when $\delta >0$ is sufficiently small, we conclude that for all $a \in \mathbb{R}$
    \begin{align*}
        \lim_{\delta \to 0}\mathbb{P}_{\delta}\left(\int_{0}^{\infty}X_{s}ds + \frac{4}{7}\ln\delta\ge a\right) &= \lim_{\delta \to 0}\text{erf}\left(\frac{\sqrt{7}\delta^{-1}}{2\sqrt{\exp(7a/2 - 2 \ln\delta)-1}} \right)
        = \text{erf}\left(\frac{\sqrt7}2e^{-7a/4}\right). 
    \end{align*}
We now turn to the general case with $\alpha \in \mathbb{R}$, where we abbreviate $\mathbb{P}_{2\alpha,7/4,1,1/\delta}$ with $\mathbb{P}_{\alpha,\delta}$.
Furthermore, let $R$ be an Ornstein-Uhlenbeck process that under $\mathbb{P}_{\alpha,\delta}$ satisfies the SDE
\[ dR_{s} = \left(2\alpha- \frac{7}{4}R_{s}\right)ds + dB_{s},\; s > 0, \quad\text{with}\quad R_{0}=\delta^{-1}.\]
Then, the shifted process $\tilde{R} = R - 8\alpha/7$ is also an O-U process under $\mathbb{P}_{\alpha,\delta}$, since 
\[  d\tilde{R}_{s} = - \frac{7}{4}\tilde{R}_{s}ds + dB_{s},\; s > 0, \quad\text{with}\quad \tilde{R}_0 = \delta^{-1}-8\alpha/7. \]
If we define the stopping time 
\[\tau_{x} = \inf\{s \ge 0 : \tilde{R}_{s} = -8x/7\}, \quad x \in \mathbb{R},\]
then $\tau_{\alpha} =T(R) = \inf\{s \ge 0 : R_{s} = 0\}$ under $\mathbb{P}_{\alpha,\delta}$. So, Lemma~\ref{lem:integralequalsTprime} 
implies that
$\int_{0}^{\infty}X_{s}ds$ and $\tau_{\alpha}$ have the same distribution under $\mathbb{P}_{\alpha,\delta}$. Further, we obtain analogously to \eqref{eq:distribution Integral} that 
\[\mathbb{P}_{\alpha,\delta}\left(\tau_{0} + \frac{4}{7} \ln\delta \ge a\right) = \text{erf}\left(\frac{\sqrt{7}\delta^{-1} -8\alpha/\sqrt{7}}{2\sqrt{\exp(7a/2 - 2 \ln\delta)-1}} \right) , \quad a \in \mathbb{R}. \]
Hence, for all $a\in \mathbb{R}$ 
\begin{align*}
\mathbb{P}_{\alpha,\delta}\left(\int_{0}^{\infty}X_{s}ds + \frac{4}{7}\ln\delta \ge a \right)
&=\mathbb{E}_{\alpha, \delta}\left[\mathbb{P}_{\alpha,\delta}\left(\tau_{0} + \frac{4}{7}\ln\delta \ge a  - (\tau_{\alpha}-\tau_{0})  \mid \tau_{\alpha}-\tau_{0} \right) \right]\\
&= \mathbb{E}_{\alpha, \delta}\left[ \text{erf}\left(\frac{\sqrt{7}\delta^{-1} -8\alpha/\sqrt{7}}{2\sqrt{\exp(7(a-(\tau_{\alpha}-\tau_{0}))/2 - 2 \ln\delta)-1}} \right)\right]  .
\end{align*}
Note that $\tau_{0}-\tau_{\alpha}$ is independent of $\delta$ for $\delta$ small enough, as for $\alpha \le0$, the difference $\tau_{0}-\tau_{\alpha}$ represents the time the process $\tilde{R}$ started in $-8\alpha/7$ needs to hit $0$. For $\alpha\ge 0$, this observation follows by a very similar argument for $\tau_{\alpha}-\tau_{0}$. Together with the dominated convergence theorem, we therefore obtain
\begin{align*}
    & \lim_{\delta \to 0}\mathbb{E}_{\alpha,\delta}\left[ \text{erf}\left(\frac{\sqrt{7}\delta^{-1} -8\alpha/\sqrt{7}}{2\sqrt{\exp(7(a-(\tau_{\alpha}-\tau_{0}))/2 - 2 \ln\delta)-1}} \right)\right]\\
    = & \lim_{\delta \to 0}\mathbb{E}_{\alpha,0}\left[ \text{erf}\left(\frac{\sqrt{7}\delta^{-1} -8\alpha/\sqrt{7}}{2\sqrt{\exp(7(a-(\tau_{\alpha}-\tau_{0}))/2 - 2 \ln\delta)-1}} \right)\right]
    = \mathbb{E}_{\alpha,0}\left[\text{erf}\left(\frac{\sqrt{7}}{2e^{7(a - (\tau_{\alpha}-\tau_{0}))/4}}\right) \right],
\end{align*}
from which the claim readily follows. 
\end{proof}
We are now in  a position to complete  the proof of Theorem~\ref{numbjumps}. Since \eqref{eq:earlySteps} occurs with probability at least $1-\delta$,  
\begin{align*}
    \mathbb{P}\Bigg(n^{-1}\sum_{t\ge 0} U_{t}- \frac{2}{7}\ln n \le a \Bigg)
    \le \mathbb{P}\Bigg(n^{-1}\sum_{t> T_{n,M,\delta}}U_{t} + \frac{4}{7}\ln \delta + \chi - \epsilon \le a \Bigg) + \delta, \quad a\in \mathbb{R}.
\end{align*}
Subsequently, we again abbreviate $\mathbb{P}_{2\alpha, 7/4, 1, 1/\delta}$ with $\mathbb{P}_{\alpha,\delta}$. Since $\lim_{n\to \infty}n^{-1}U_{T_{n,M,\delta}}=0$, it follows from Lemma~\ref{lem:convergenceSumToIntegral} that
\[\lim_{n\to \infty}\mathbb{P}\Bigg(n^{-1}\sum_{t> T_{n,M,\delta}}U_{t} + \frac{4}{7}\ln \delta + \chi - \epsilon \le a \Bigg) = \mathbb{P}_{\alpha,\delta}\left(\int_{0}^{\infty}X_{s}ds + \frac47\ln \delta + \chi - \epsilon \le a \right),\quad a \in \mathbb{R}.\]
Letting $\delta \to 0$, Lemma~\ref{lem:convergenceIntegral} thus implies that 
\[\lim_{n\to \infty}\mathbb{P}\Bigg(n^{-1}\sum_{t\ge 0} U_{t}- \frac{2}{7}\ln n \le a \Bigg) \le \mathbb{P}_{\alpha,0}\left(\tilde{A}_{\alpha} + \chi - \epsilon \le a \right), \quad a \in \mathbb{R}.  \]
Using a very similar argument, we also estimate that 
\[\lim_{n\to \infty}\mathbb{P}\Bigg(n^{-1}\sum_{t\ge 0} U_{t}- \frac{2}{7}\ln n \ge a \Bigg) \le \mathbb{P}_{\alpha,0}\left(\tilde{A}_{\alpha} + \chi + \epsilon \ge a \right), \quad a \in \mathbb{R}.  \]
Therefore, using continuity of $\tilde{A}_{\alpha}$ and letting $\epsilon \to 0$, 
 \eqref{eq:claim} and, by employing again Lemma~\ref{lem:convergenceIntegral} for the case $\alpha =0$, the statement of the theorem readily follows.


\bibliographystyle{plain}
\bibliography{ref}

\begin{thebibliography}{10}

\bibitem{alili2005representations}
L.~Alili, P.~Patie, and J.~L. Pedersen.
\newblock Representations of the first hitting time density of an
  ornstein-uhlenbeck process.
\newblock {\em Stochastic Models}, 21(4):967--980, 2005.

\bibitem{Boucheron2004}
S.~Boucheron, G.~Lugosi, and O.~Bousquet.
\newblock {Concentration inequalities}.
\newblock In {\em {Advanced Lectures on Machine Learning}}, pages 208--240.
  Springer, 2004.

\bibitem{10.1007/978-3-662-43948-7_21}
K.~Bringmann, F.~Kuhn, K.~Panagiotou, U.~Peter, and H.~Thomas.
\newblock Internal {DLA}: Efficient simulation of a physical growth model.
\newblock In {\em Automata, Languages, and Programming}, pages 247--258,
  Berlin, Heidelberg, 2014. Springer Berlin Heidelberg.

\bibitem{CMRRS18}
C.~Cooper, A.~McDowell, T.~Radzik, N.~Rivera, and T.~Shiraga.
\newblock Dispersion processes.
\newblock {\em Random Structures Algorithms}, 53(4):561--585, 2018.

\bibitem{DMP23}
U.~De~Ambroggio, T.~Makai, and K.~Panagiotou.
\newblock Dispersion on the complete graph, 2023.
\newblock arXiv:2306.02474. An extended abstract appeared in the
  \emph{Proceedings of EUROCOMB '23}.

\bibitem{book:deBruijn}
N.~G. de~Bruijn.
\newblock {\em Asymptotic Methods in Analysis}.
\newblock Bibliotheca mathematica. Dover Publications, 1981.

\bibitem{DF91}
P.~Diaconis and W.~Fulton.
\newblock A growth model, a game, an algebra, {L}agrange inversion, and
  characteristic classes.
\newblock {\em Rend. Sem. Mat. Univ. Politec. Torino}, 49(1):95--119, 1991.

\bibitem{durrett2018stochastic}
R.~Durrett.
\newblock {\em Stochastic calculus -- a practical introduction}.
\newblock Probability and Stochastics Series. CRC Press, Boca Raton, FL, 1996.

\bibitem{foucart2020entrance}
C.~Foucart, P.~Li, and X.~Zhou.
\newblock On the entrance at infinity of {F}eller processes with no negative
  jumps.
\newblock {\em Statist. Probab. Lett.}, 165:108859, 9, 2020.

\bibitem{FP18}
A.~Frieze and W.~Pegden.
\newblock A note on dispersing particles on a line.
\newblock {\em Random Structures Algorithms}, 53(4):586--591, 2018.

\bibitem{fu2010stochastic}
Z.~Fu and Z.~Li.
\newblock Stochastic equations of non-negative processes with jumps.
\newblock {\em Stochastic Processes and their Applications}, 120(3):306--330,
  2010.

\bibitem{kent1980eigenvalue}
J.~T. Kent.
\newblock Eigenvalue expansions for diffusion hitting times.
\newblock {\em Zeitschrift f{\"u}r Wahrscheinlichkeitstheorie und Verwandte
  Gebiete}, 52(3):309--319, 1980.

\bibitem{lambert2005branching}
A.~Lambert.
\newblock The branching process with logistic growth.
\newblock {\em Ann. Appl. Probab.}, 15(2):1506--1535, 2005.

\bibitem{LBG92}
G.~F. Lawler, M.~Bramson, and D.~Griffeath.
\newblock Internal diffusion limited aggregation.
\newblock {\em Ann. Probab.}, 20(4):2117--2140, 1992.

\bibitem{roch_mdp_2024}
S.~Roch.
\newblock {\em Modern Discrete Probability: An Essential Toolkit}.
\newblock Cambridge Series in Statistical and Probabilistic Mathematics.
  Cambridge University Press, 2024.

\bibitem{ar:rollaARW}
L.~T. Rolla.
\newblock {Activated Random Walks on $\mathbb{Z}^{d}$}.
\newblock {\em Probability Surveys}, 17:478 -- 544, 2020.

\bibitem{S20}
Y.~Shang.
\newblock Longest distance of a non-uniform dispersion process on the infinite
  line.
\newblock {\em Inform. Process. Lett.}, 164:106008, 5, 2020.

\end{thebibliography}

\end{document}